\documentclass[10pt,a4paper]{article}
\usepackage[utf8]{inputenc}
\usepackage[english]{babel}
\usepackage{amsmath}
\usepackage{amsfonts}
\usepackage{amssymb}
\usepackage[english]{babel}
\usepackage{tikz-cd}
\usepackage{amsthm}
\usepackage{mathrsfs}
\usepackage{xfrac}
\usepackage{mathtools}
\usepackage{extarrows}
\usetikzlibrary{arrows}
\usepackage{array}
\usepackage{url}
\usepackage{geometry}
\usepackage{sectsty}
\usepackage{hyperref}
\usepackage{enumerate}
\usepackage{enumitem}
\usepackage{caption}
\usepackage{subcaption}
\numberwithin{equation}{section}
\author{Quentin Faes}
\title{Triviality of the $J_4$-equivalence  among \\homology 3-spheres}
{\theoremstyle {definition} \newtheorem {defi} {Définition} [section] }
\newtheorem{definition}[defi]{Definition}
\newtheorem{theorem}[defi]{Theorem}

\newtheorem{coroll}[defi]{Corollary}
\newtheorem{lemma}[defi]{Lemma}
\newtheorem{propo}[defi]{Proposition}
\theoremstyle{definition}
\newtheorem*{ack}{Acknowledgements}
\newtheorem{rem}[defi]{Remark}

\theoremstyle{plain}
\newtheorem{thlett}{Theorem}

\usetikzlibrary{arrows}
\usepackage{indentfirst}
\makeatletter
\def\blfootnote{\xdef\@thefnmark{}\@footnotetext}
\makeatother
\usepackage{tocloft, calc}
\usepackage[titletoc]{appendix}

\usepackage{titlesec}
\usepackage{changepage}

\newcommand{\address}{{% additional braces for segregating \footnotesize
  \bigskip
  \footnotesize
  \textsc{Institut de mathématiques de Bourgogne, UMR 5584, Université Bourgogne Franche-Comté, 21000 Dijon, France}\par\nopagebreak
  \textit{E-mail address}  \texttt{quentin.faes@u-bourgogne.fr}}}
\usepackage[runin]{abstract}
\usepackage{listings}
\abslabeldelim{.}
\definecolor{wqwqwq}{rgb}{0,0,0}
\usepackage{pgf,tikz,pgfplots}

\newcommand{\ltree}[4]{\tikz[baseline=8pt, scale = 0.5]{
\draw [color=wqwqwq] (0,2-1.25)-- (2,2-1.25);
\draw [color=wqwqwq] (0,1.5)-- (0,0);
\draw [color=wqwqwq] (2,1.5)-- (2,0);

\draw[color=wqwqwq] (0,3-1.25) node {#1};
\draw[color=wqwqwq] (0,1.1-1.25) node {#2};
\draw[color=wqwqwq] (2,1.1-1.25) node {#3};
\draw[color=wqwqwq] (2,3-1.25) node {#4};}}

\newcommand{\lfivetree}[5]{\tikz[baseline=8pt ,scale = 0.5]{
\draw [color=wqwqwq] (0,1.25)-- (0,0.25);
\draw [color=wqwqwq] (0,0.75)-- (2,0.75);
\draw [color=wqwqwq] (2,1.25)-- (2,0.25);
\draw [color=wqwqwq] (1,0.75)-- (1,1.5);
\draw[color=wqwqwq] (0,0) node {#1};
\draw[color=wqwqwq] (0,1.6) node {#2};
\draw[color=wqwqwq] (1,1.75) node {#3};
\draw[color=wqwqwq] (2,1.6) node {#4};
\draw[color=wqwqwq] (2,0) node {#5};}
}
\newcommand{\bfivetree}[5]{\tikz[baseline=17pt ,scale = 0.9]{
\draw [color=wqwqwq] (0,1.25)-- (0,0.25);
\draw [color=wqwqwq] (0,0.75)-- (2,0.75);
\draw [color=wqwqwq] (2,1.25)-- (2,0.25);
\draw [color=wqwqwq] (1,0.75)-- (1,1.7);
\draw[color=wqwqwq] (0,0) node[scale = 0.8]{#1};
\draw[color=wqwqwq] (0,1.6) node[scale = 0.8] {#2};
\draw[color=wqwqwq] (1,2) node[scale = 0.8] {#3};
\draw[color=wqwqwq] (2,1.6) node[scale = 0.8] {#4};
\draw[color=wqwqwq] (2,0) node[scale = 0.8] {#5};}
}
\newcommand{\odottree}[2]{#1 \odot #2}

\newcommand{\ltritree}[3]{\tikz[baseline=8pt ,scale = 0.5]{
\draw [color=wqwqwq] (1,1.5)-- (1,0.75);
\draw [color=wqwqwq] (1,0.75)-- (1-0.866*0.75,0.75-.37);
\draw [color=wqwqwq] (1,0.75)-- (1+0.866*0.75,0.75-.37);

\draw[color=wqwqwq] (1,1.5+0.2) node {#1};
\draw[color=wqwqwq] (1-0.866*0.75-0.2,0.75-.37-0.2) node {#2};
\draw[color=wqwqwq] (1+0.866*0.75+0.2,0.75-.37-0.2) node {#3};}}

\definecolor{codegreen}{rgb}{0,0.6,0}
\definecolor{codegray}{rgb}{0.5,0.5,0.5}
\definecolor{codepurple}{rgb}{0.58,0,0.82}
\definecolor{backcolour}{rgb}{0.95,0.95,0.92}

\lstdefinestyle{mystyle}{
    backgroundcolor=\color{backcolour},   
    commentstyle=\color{codegreen},
    keywordstyle=\color{magenta},
    numberstyle=\tiny\color{codegray},
    stringstyle=\color{codepurple},
    basicstyle=\ttfamily\footnotesize,
    breakatwhitespace=false,         
    breaklines=true,                 
    captionpos=b,                    
    keepspaces=true,                 
    numbers=left,                    
    numbersep=5pt,                  
    showspaces=false,                
    showstringspaces=false,
    showtabs=false,                  
    tabsize=2
}

\newcommand{\bracket}[2]{\Bigg[#1;#2\Bigg]}

\newcommand{\Q}{\mathbb{Q}}
\newcommand{\Z}{\mathbb{Z}}

\newcommand{\K}{\mathcal{K}}
\newcommand{\kp}{\mathcal{K}'}
\newcommand{\kpp}{\mathcal{K}''}
\newcommand{\I}{\mathcal{I}}

\newcommand{\M}{\mathcal{M}}
\newcommand{\sph}{\operatorname{Sp}(H)}

\newcommand{\ses}[3]{\[ 0 \longrightarrow #1 \longrightarrow #2 \longrightarrow #3 \longrightarrow 0\]}

\begin{document}
\maketitle
\titlelabel{\thetitle.  }
\begin{adjustwidth}{20 pt}{20pt}
\textsc{Abstract.}  We prove that all homology 3-spheres are $J_4$-equivalent, i.e. that any homology 3-sphere can be obtained from one another by twisting one of its Heegaard splittings by an element of the mapping class group acting trivially on the fourth nilpotent quotient of the fundamental group of the gluing surface. We do so by exhibiting an element of $J_4$, the fourth term of the Johnson filtration of the mapping class group, on which (the core of) the Casson invariant takes the value $1$. In particular, this provides an explicit example of an element of $J_4$ that is not a commutator of length $2$ in the Torelli group.
\end{adjustwidth}
\setcounter{tocdepth}{1}
\renewcommand{\contentsname}{\hfill Contents\hfill}
\renewcommand{\cfttoctitlefont}{\scshape}
\renewcommand{\cftaftertoctitle}{\hfill}
\renewcommand{\cftsecaftersnum}{.}
\setlength{\cftbeforesecskip}{0.05cm}
\tableofcontents
\section{Introduction}
The study of the mapping class groups of surfaces is related to the study of 3-manifolds through the notion of \emph{Heegaard splitting}, a way of presenting a 3-manifold by specifying a gluing map between the boundaries of two handlebodies. Also, one can take such a presentation and compose the gluing map by an element of the mapping class group, yielding another 3-manifold. This procedure of ``twisting'' Heegaard splittings is called a ``surgery''. When we restrict the surgeries to certain subgroups of the mapping class group, it provides us with some equivalence relations among the set of homeomorphism classes of 3-manifolds. This can help to understand and study the topological properties of 3-manifolds invariants, in particular the ones of the family of so-called ``finite-type invariants'' (see, e.g., \cite{mor} for the case of the Casson invariant). For example, one can consider restricting surgeries to the \emph{lower central series} or the \emph{Johnson filtration} of the \emph{Torelli group}. Let $\Sigma := \Sigma_{g,1}$ be a compact connected oriented surface of genus $g$, with one boundary component, and $\mathcal{M}= \mathcal{M}(\Sigma)$ the mapping class group of $\Sigma$ relative to the boundary. The Torelli group $\I = \I(\Sigma)$ is the subgroup of $\M$ consisting of elements acting trivially on the first homology group of the surface. Let $k$ be a positive integer. The $(k+1)$-th term of the lower central series of the Torelli group is defined inductively as the commutator subgroup $\Gamma_{k+1}\I := [\Gamma_k \I , \I]$, with $\Gamma_1\I :=  \I$. The $k$-th term $J_k= J_k(\Sigma)$ of the Johnson filtration consists of elements acting trivially on the $k$-th nilpotent quotient of the fundamental group of $\Sigma$.
 \blfootnote{This research has been supported by the project “AlMaRe” (ANR-19-CE40-0001-01) of the ANR and the project ``ITIQ-3D" of the Région Bourgogne Franche-Comté.

The IMB receives support from the EIPHI Graduate School (contract ANR-17-EURE-0002).}

Let us denote by $\mathcal{S}(3)$ the set of homeomorphism classes of homology 3-spheres (to which we restrict ourselves in this paper). For any embedding $j$ of the surface $\Sigma$ in a homology 3-sphere $M$, and for any element $\varphi$ of the Torelli group of the surface, one can define a new homology 3-sphere $M_{j,\varphi}$ by removing from $M$ an open neighborhood $\nu \Sigma$ of $\Sigma$, and gluing the mapping cylinder of $\varphi$ in place of $\nu \Sigma$. It is known that any homology 3-sphere can be obtained from $S^3$ in this way (see, for example, \cite{mat87}). Two homology 3-spheres $M$ and $M'$ are said to be \emph{$J_k$-equivalent} (resp \emph{$Y_k$-equivalent}) if $M'$ is homeomorphic to $M_{j, \phi}$ for some embedding $j : \Sigma \rightarrow M$ of a surface $\Sigma$ and some $\phi \in J_k(\Sigma)$ (resp. in $\Gamma_k \I(\Sigma)$). These are known to be equivalence relations, and one can easily see that $Y_k$-equivalence implies $J_k$-equivalence. Moreover, in the above definitions, we can always restrict ourselves to the case where the boundary of $M \setminus \nu\Sigma$ is a Heegaard surface, thus limiting ourselves to the above-mentioned notion of ``surgery'' (see \cite[Lemma 2.1]{masY3}, for instance).  Morita \cite{mor} and Pitsch \cite{pit}, successively, have shown that $J_2$-equivalence and $J_3$-equivalence are trivial on $\mathcal{S}(3)$. In constrast, the $Y_2$-equivalence is non-trivial and is classified by Rochlin's invariant $\mu: \mathcal{S}(3) \rightarrow \Z / 2\Z$ (see \cite{habclasper} and\cite{masY2cl}). The goal of this paper is to show that $J_4$-equivalence is also trivial on $\mathcal{S}(3)$.

Let us now be more precise. We recall the definition of the Johnson filtration and the \emph{Johnson homomorphisms}, which have been introduced and studied by Johnson and Morita in \cite{johsurvey, mor93}. Recall that $\pi := \pi_1(\Sigma)$ is a free group. We denote by $\zeta$ the element of $\pi$ corresponding to the oriented boundary $\partial \Sigma$. For $k \geq 1$, we consider the lower central series of the group $\pi$, the filtration $(\Gamma_k \pi)_{k \geq 1}$. We call the quotient $N_k := \pi / \Gamma_{k+1} \pi$ the \emph{$k$-th nilpotent quotient} of $\pi$. The first nilpotent quotient is canonically isomorphic to $H := H_1(\Sigma, \Z)$. The intersection form of the surface induces a symplectic form $\omega$ on the abelian group $H$. The action of $\M$ on the surface induces the natural $\sph$-module structure on $H$, as any transformation of the surface preserves the intersection form. The curves $(\alpha_i)_{1 \leq i \leq g}$ and $(\beta_i)_{1 \leq i \leq g}$ on Figure \ref{surface} are two cut systems of $\Sigma$ such that each curve in the first one has exactly one intersection point with exactly one curve in the second one, and vice versa. Such a choice is called a system of \emph{meridians} and \emph{parallels}. In particular, it fixes a choice of a symplectic basis for $H = \mathbb{Z}\langle a_1,a_2 \dots a_g, b_1, b_2, \dots b_g\rangle$, where $a_i$ (resp. $b_i$) is the homology class of $\alpha_i$ (resp. $\beta_i$).

\begin{figure}[h]
	\centering
	\includegraphics[scale= 0.23]{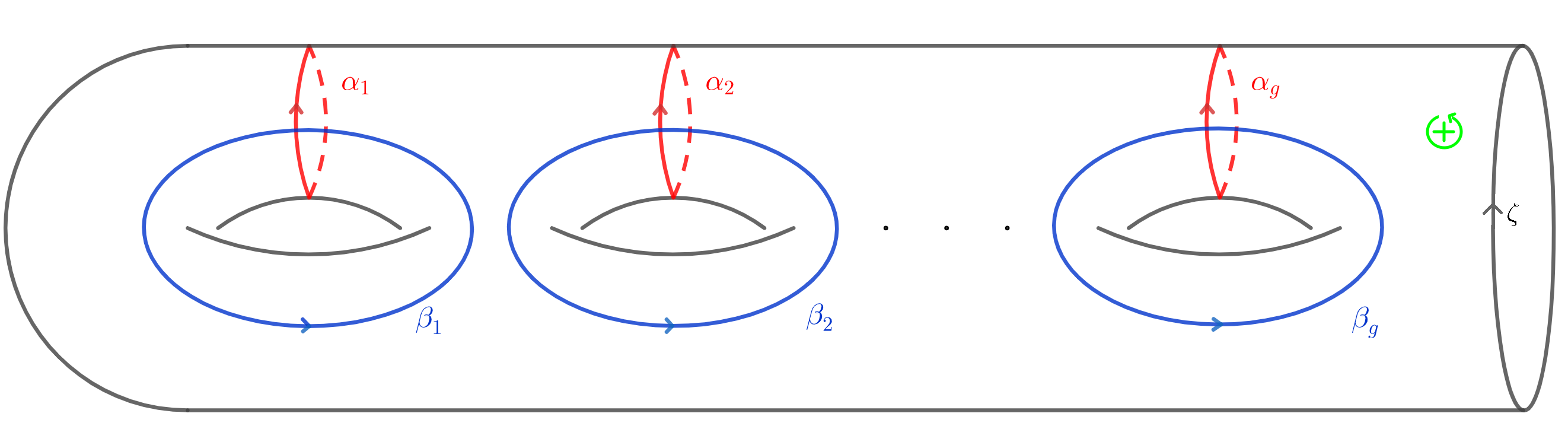}
	\caption{Model for $\Sigma_{g,1}$, and a possible choice of system of meridians and parallels}
	\label{surface}
\end{figure}

We will also need at some point to lift these curves to a basis of $\pi= \pi_1(\Sigma, x_0)$, with $x_0 \in \partial \Sigma$. We do so as described in Figure \ref{basispi}.

\begin{figure}[h]
	\centering
	\includegraphics[scale= 0.21]{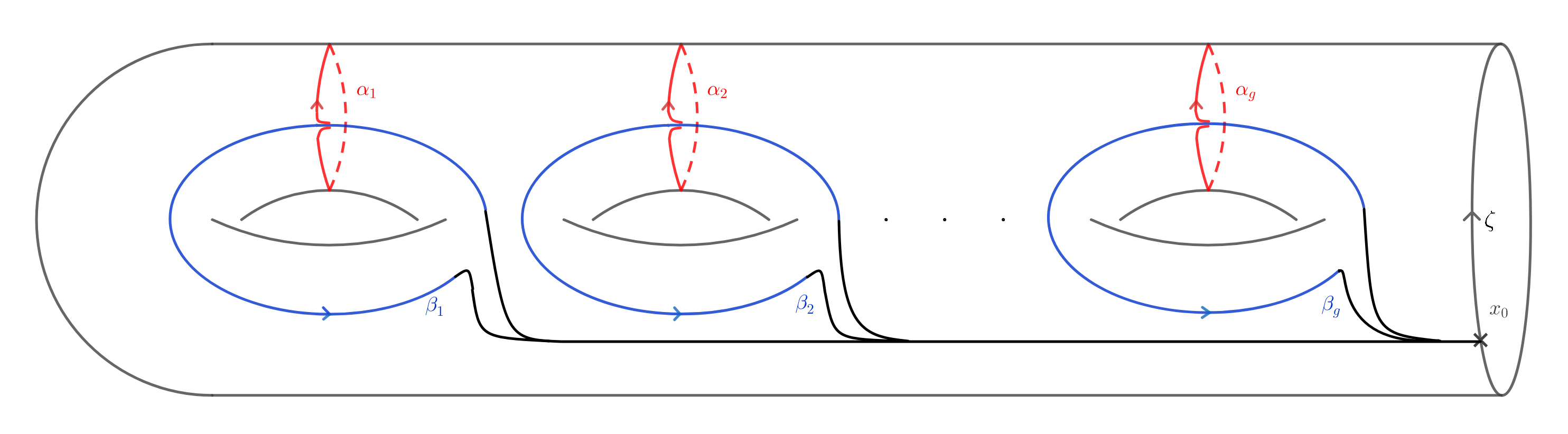}
	\caption{Based curves inducing a basis of $\pi$}
	\label{basispi}
\end{figure}
It is clear that $\mathcal{M}$ acts both on $\pi$ and all its nilpotent quotients. There is an exact sequence: \ses{\mathcal{L}_{k+1}(H)}{N_{k+1}}{N_{k}} where $\mathcal{L}(H)$ stands for the graded free Lie algebra generated by $H$ in degree $1$, and the first non-trivial arrow is given by the identification between $\mathcal{L}_{k+1}(H)$ and $\Gamma_{k+1} \pi / \Gamma_{k+2} \pi$. This sequence induces the short exact sequence:
\[ 0 \longrightarrow \operatorname{Hom}(H,\mathcal{L}_{k+1}(H)) \longrightarrow \operatorname{Aut}(N_{k+1}) \longrightarrow \operatorname{Aut}(N_{k}). \]
 \noindent The group $J_k$ is defined as the kernel of the canonical homomorphism $\rho_k : \mathcal{M} \rightarrow \operatorname{Aut}(N_k)$. In particular, by the Hurewicz theorem, $J_1$ is the Torelli group, otherwise denoted $\mathcal{I} = \I(\Sigma)$. The alternative notation $\mathcal{K} = \K(\Sigma)$ is also sometimes used for $J_2$. The restriction of $\rho_{k+1}$ to $J_k$ then induces a morphism: \[ \tau_k : J_k \longrightarrow \operatorname{Hom}(H,\mathcal{L}_{k+1}(H)). \]
This map is called the \textit{$k$-th Johnson homomorphism}, and its kernel is $J_{k+1}$. Furthermore, the mapping class group acts on itself by conjugation, inducing an action of the symplectic group $\operatorname{Sp}(H)$ on the quotient $J_k/ J_{k+1}$. Each $\tau_k$ is then $\operatorname{Sp}(H)$-equivariant. It is also known that the graded space associated to the Johnson filtration has a Lie structure, its bracket being induced by the commutator in $\mathcal{M}$. The target space of $\tau_k$ can be identified with the space of \textit{ derivations of degree $k$}, i.e. derivations of the Lie algebra $\mathcal{L}(H)$ mapping $H = \mathcal{L}_1(H)$ to $\mathcal{L}_{k+1}(H)$. We denote by $D_k(H)$ the subspace of \textit{symplectic derivations} of degree $k$, which consists of derivations of $\mathcal{L}(H)$ of degree $k$ sending $\tilde{\omega} \in \Lambda^2 H \simeq  \mathcal{L}_2(H)$, the bivector dual to $\omega$, to 0. The fact that an element of $\mathcal{M}$ fixes the boundary of $\Sigma$ allows to further restrict the target of $\tau_k$ to $D_k(H)$. Also, $D_k(H)$ can be defined by the short exact sequence: \ses{D_k(H)}{H \otimes \mathcal{L}_{k+1}(H)}{\mathcal{L}_{k+2}(H)} \noindent where the arrow from ${H \otimes \mathcal{L}_{k+1}(H)}$ to ${\mathcal{L}_{k+2}(H)}$ is the bracket of the free Lie algebra.

With these definitions, the spaces $(D_k(H))_{k \geq 1}$ reassembles in a graded Lie algebra $D(H)$ (the bracket of two derivations $d_1$ and $d_2$ being classically defined as $ d_1 d_2 -d_2 d_1$). The family $(\tau_k)_{k \geq 1}$ induces a map:
\[ \tau : \underset{k\geq 1}{\bigoplus} J_k / J_{k+1} \longrightarrow D(H)\]
\noindent which is an $\operatorname{Sp}(H)$-equivariant graded Lie morphism.

The Casson invariant $\lambda: \mathcal{S}(3) \rightarrow \Z$ is an invariant of homology 3-spheres, lifting the Rochlin invariant $\mu : \mathcal{S}(3) \rightarrow \mathbb{Z}_2$, and originally defined by counting in some way the number of irreducible representations of the fundamental group of the homology 3-sphere into $SU(2)$. In \cite{mor}, Morita explained that the map $\lambda_j$ induced on $\I$ by the Casson invariant and any Heegaard embedding $j : \Sigma \rightarrow S^3$: \begin{align*}
\lambda_j : \mathcal{I} &\longrightarrow \mathbb{Z} \\
\varphi & \longmapsto \lambda(S^3_{j,\varphi})
\end{align*}
is \emph{not} a homomorphism (here, by a Heegaard embedding, we mean that capping the surface $j(\Sigma)$ by a disk yields a Heegaard splitting of $S^3$). Nevertheless, Morita showed that its restriction to $\mathcal{K} :=J_2$ is a homomorphism. He also showed that this restriction can be expressed as the sum of two homomorphisms: 

\begin{equation} -\lambda_{j}=\frac{1}{24} d+q_{j}: \mathcal{K} \rightarrow \mathbb{Z}.\label{eqmorita}\end{equation}

\noindent The map $d$, named \emph{the core of the Casson invariant} is independent of $j$, and the map $q_j$  factorizes through the second Johnson homomorphism. Thus, the Casson invariant induces a well-defined homomorphism $\lambda$ on $J_k$ for any $k \geq 3$, meaning that the value of the Casson invariant on $S^3_{j,\varphi}$ is independent of $j$ when $k \geq3$. The map $d$ is rather difficult to understand, but it is known that Dehn twists along bounding simple closed curves (abbreviated BSCC in the rest of this paper) of genus 1 and 2 generate $\K$ and that the value of $d$ on a Dehn twist along a BSCC of genus $h$ is $4h(h-1)$. Recall that the genus of a BSCC is defined as the genus of the subsurface bounding the curve which does not contain $\partial \Sigma$. We will denote $\kp$ (resp. $\kpp$) the subgroup of $\K$ generated by twists around BSCC of genus 1 (resp. genus~2).

It was claimed by Morita, also in \cite{mor}, and written explicitly by Massuyeau and Meilhan in \cite{masY3}, that $\lambda(J_3)= \mathbb{Z}$ in genus $g=2$. Moreover, according to Habiro \cite{habclasper}, $Y_3$-equivalence among homology 3-spheres is classified by $\lambda$. Massuyeau and Meilhan \cite[Rem. 6.4]{masY3} then explained how to reprove from these two facts Pitsch's result stating that any two homology 3-spheres are always $J_3$-equivalent. Using the same strategy, we shall prove the following.

\begin{thlett}
For any genus $g \geq 2$, the restriction of $\lambda : \K \rightarrow \mathbb{Z}$ to $J_4$ is surjective.
\label{thA}
\end{thlett}

\begin{thlett}
The $J_4$-equivalence is trivial on $\mathcal{S}(3)$. In other words, every homology 3-sphere can be obtained from $S^3$ by twisting one of its Heegaard splitting by an element of the fourth term of its Johnson filtration.
\label{thB}
\end{thlett}

A motivation for Theorem \ref{thA} is given by the Dehn-Nielsen Theorem, which states in particular that an element of the mapping class group $\M$ is completely determined by its action on $\pi$. Thus, given a map on $\M$, one could hope to compute it by purely algebraic methods, considering only this action. This is also a motivation when using the Johnson homomorphisms. For the case of the Casson invariant, we already know that there is no $k \geq 1$ such that $\lambda$ can be fully computed from $\rho_k : \M \rightarrow \operatorname{Aut}(N_k)$, as Hain~\cite{haininf} proved that $\lambda(J_k)~\neq~\{ 0 \}$ for $k \geq 3$. Nevertheless, it is still possible that for some $k \geq 5$, the homomorphism $\lambda$ restricted to $J_k$ is no longer surjective. This question is of course related to the determination of the graded space associated to the Johnson filtration.

In order to prove Theorem \ref{thA}, we will explicitly build an element $\varphi$ in $J_4$ whose Casson invariant is equal to $1$. Remarkably, this element also seems to be the first explicit example of an element of $J_4$ which cannot be obtained as a bracket of elements of lower terms of the Johnson filtration. The result of Hain \cite{haininf} mentioned above implies that there are elements which are not in the commutator subgroup $[\I, \K]$ arbitrarily far down in the Johnson filtration. Theorem \ref{thB} is next deduced from Theorem \ref{thA}, using the classification of the $Y_4$-equivalence by Habiro \cite{habclasper}.

The paper is organized as follows. In Section \ref{sec2}, we recall and use computational tools developed by Morita in \cite{mor} and Kawazumi and Kuno in \cite{KK}, and we build the element $\varphi$. We also discuss the complexity of the construction of such an element. In Section \ref{sec3}, we define $J_4$-equivalence and explain how the fact that $\lambda(J_4)= \mathbb{Z}$ implies the triviality of $J_4$-equivalence on $\mathcal{S}(3)$. Finally, the computer program used to claim some equalities in Section \ref{sec2} is given and explained in appendix \ref{appA}.

\begin{ack}
I would like to thank my advisor, Gwénaël Massuyeau, for guiding and supporting me throughout this work, the members of the forum ask.sagemaths.org for their answers to my questions, and the reviewer for his thoughtful comments.  \end{ack}
\label{intro}

\section{A special element of $J_4$}
\label{sec2}
In this section, we recall a method presented by Kawazumi and Kuno \cite{KK} to explicitly compute the action on $\pi$ of any element of the Torelli group and, in particular, the image of an element in $J_k$ by the $k$-th Johnson homomorphism. We then use this method to compute $\tau_3(\psi)$, where $\psi \in \M$ is a certain element of $J_3$ such that $\lambda(\psi) =1$. This element was first presented in \cite{masY3}. We then prove that $\tau_3(\psi) \in \tau_3([\K,\I])$, and use this to build an element $\varphi \in J_4$ such that $\lambda(\varphi)=1$, proving Theorem \ref{thA}.

\subsection{The Kawazumi-Kuno formula}
All the results used here are from \cite{KK}, but the reader will find enlightening additional information in \cite{MasTur}. We refer to these papers for further details.

We denote by $H_\Q := H \otimes \Q$ the rationalization of $H$ and by $\widehat{T} :=\prod_{m=0}^{\infty} H_\Q^{\otimes m}$ the completed tensor algebra generated by $H_\Q$. The algebra $\widehat{T}$ is filtered by the sequence of two-sided ideals $\widehat{T}_p :=\prod_{m \geq p}^{\infty} H_\Q^{\otimes m}$. It is known that the choice of a Magnus expansion (in the sense of \cite{kawa}), gives an identification of $\widehat{\Q\pi}$, the completed group algebra of $\pi$ (with respect to the filtration induced by the augmentation ideal), with $\widehat{T}$. Furthermore, Massuyeau introduced in \cite{masinf} the notion of \emph{symplectic expansion}. Shortly, a Magnus expansion is a monoid map $\theta$ from $\pi$ to $\widehat{T}$ such that $\theta(x) = 1 + \{x \} + deg_{\geq 2} \text{ for all } x \in \pi$, where $\{x \}$ is the class of $x$ in $H_\Q$. A symplectic expansion is then an expansion taking group-like values in $\widehat{T}$ and sending $\zeta$ to $\operatorname{exp}(-\tilde{\omega})$, where $\tilde{\omega}$ is the element of $H_\Q^{\otimes 2}$ representing the symplectic form through Poincaré duality.

As any element of $\M$ acts naturally on $\widehat{\Q \pi}$, a Magnus expansion provides, via the identification $\widehat{\Q\pi} \simeq \widehat{T}$, a map $T^\theta : \M \rightarrow \operatorname{Aut}(\widehat{T})$. In particular, denoting by $T_\gamma$ the Dehn twist on a simple closed curve $\gamma$ on $\Sigma$, the map $T^\theta(T_\gamma)$ is an automorphism of $\widehat{T}$. We now fix a Magnus expansion, and we call $T^\theta$ the \emph{total Johnson map}, this terminology being justified by the following theorem, derived from \cite[Theorem 3.1]{kawa}. Define, for any $\phi \in \I$, $\tau^\theta(\phi) := T^\theta(\phi)_{\vert H_\Q} - \operatorname{Id}_{\vert H_\Q}$ and write $\tau^\theta = \sum \limits_{k \geq 1} \tau^\theta_k$, with $\tau_k^\theta(\phi) \in \operatorname{Hom}(H_\Q,H_\Q^{\otimes (k+1)})$. Also, consider the rational version $\tau_{k,\Q} := \tau_k \otimes \Q$ of the $k$-th Johnson homomorphism.

\begin{theorem}[Kawazumi]
For any Magnus expansion $\theta$, any $k \geq 1$, and any $\phi \in J_k$, we have 
$$ \tau^\theta_k(\phi) = \tau_{k,\Q}(\phi). $$ where $\mathcal{L}_{k+1}(H) \otimes \Q$ is regarded as a subspace of $H_\Q^{\otimes (k+1)}$ in the usual way.
\label{KK1}
\end{theorem}

We now define $l^\theta:= \operatorname{log}\circ \theta$, which is a map from $\pi$ to $\widehat{T}$, and consider the ``cyclicization'' map $N : \widehat{T}_1 \rightarrow \widehat{T}_1$, defined by $\left.N\right|_{H_\Q^{\otimes p}}=\sum_{m=0}^{p-1} \nu^{m},$ for $p \geq 1,$ where $\nu: H_\Q^{\otimes p} \rightarrow H_\Q^{\otimes p}$ is the map induced by the cyclic permutation of the tensors. We also define, for all $x$ in $\pi$, \begin{equation}L^\theta(x) := \frac{1}{2}N(l^\theta(x)^2) \in H_\Q \otimes \widehat{T}_1= \widehat{T}_2. \label{Ldef}\end{equation} 
The value of $L^\theta$ on $x \in \pi$ actually only depends on the conjugacy class of $x$ and $L^\theta(x) = L^\theta(x^{-1})$. Thus, we can write abusively $L^\theta(\gamma)$ for any closed curve $\gamma$ in $\Sigma$. Furthermore, identifying $H_\Q \otimes \widehat{T}_1$ with $\operatorname{Hom}(H_\Q,\widehat{T}_1)$ by Poincaré duality, such an element $L^\theta(\gamma)$ can be regarded as a weakly nilpotent derivation of $\widehat{T}$. By a \emph{weakly nilpotent} derivation, we mean a derivation $d$ such that $d(\widehat{T}_p) \subset \widehat{T}_p$, and for any $p \geq 1$, there is some $n \geq 1$ such that $d^n(\widehat{T}) \subset \widehat{T}_p$. The exponential of such a derivation is a well-defined automorphism of $\widehat{T}$. 

\begin{rem}
The reader should be aware that there are two ways to identify $H$ (or $H_\Q$) with its dual, and that the conventions chosen in \cite{kawa,KK} and \cite{masinf,MasTur} are different. We chose the convention of the latter, meaning that the isomorphism from $H$ to $H^*$ is given by $x \mapsto \omega(x,-)$.
\end{rem}
We now recall a formula due to Kawazumi and Kuno \cite[Theorem 1.1.1]{KK}.

\begin{theorem}[Kawazumi, Kuno]
Let $\theta$ be a symplectic expansion and $\gamma$ a simple closed
curve on $\Sigma$. Then we have:
\begin{equation}
T^{\theta}\left(T_{\gamma}\right)=e^{L^{\theta}(\gamma)}.
\label{eqKK}
\end{equation}
\label{KK2}
\end{theorem}

By degree considerations, making use of equation \eqref{eqKK}, Kawazumi and Kuno also get the following simple formulas, which are extracted from \cite[Theorem 6.3.1]{KK}:

\begin{coroll}[Kawazumi, Kuno]
Let $\theta$ be a symplectic expansion and $\gamma$ a bounding simple closed curve. Then we have: $$
\begin{array}{l}
\tau_{2}^{\theta}\left(T_{\gamma}\right)=L^\theta_{4}(\gamma) \\
\tau_{3}^{\theta}\left(T_{\gamma}\right)=L^\theta_{5}(\gamma).
\end{array}$$
\label{KK3}
\end{coroll}

\noindent Hence, as implied by Theorem \ref{KK1}, if $\varphi$ is given as a product of Dehn twists along BSCC and is in $J_2$ (resp. $J_3$), it is easy to compute $\tau_2(\varphi) = \tau_2^\theta(\varphi)$ (resp. $\tau_3(\varphi) = \tau_3^\theta(\varphi)$) by using Corollary \ref{KK3}.

In order to do so, we will need an instance of a symplectic expansion $\theta$. As explained in \cite{masinf}, we can build inductively such an expansion and write explicitly its values on generators of $\pi$ in low degrees. The following proposition shows that for our purposes we only need to know the values of a symplectic expansion up to order 2.

\begin{propo}
Let $\gamma$ be a bounding simple closed curve. In order to compute $L^\theta_{4}(\gamma)$ and $L^\theta_{5}(\gamma)$, we only need the expression of $l^\theta$ up to degree~$2$.
\end{propo}

\begin{proof}
As $\gamma$ is separating, any of its representative in $\pi$ will be a product of commutators. Therefore $l^{\theta}(\gamma)$ starts in degree~$2$ and it follows that we only need the expression of $l^\theta$ up to order $3$ to compute $L^\theta_{4}$ and $L^\theta_{5}$, due to the square in formula \eqref{Ldef}. Let $U,V$ be elements of $\pi$, let $u,v$ be their respective classes in $H$, and write
\begin{align*}
\theta(U) &= 1 + u + \theta_2(U) +\theta_3(U) + deg_{\geq 4} \in \widehat{T}\\
\theta(V) &= 1 + v + \theta_2(V) +\theta_3(V) + deg_{\geq 4} \in \widehat{T}
\end{align*}
where $\theta_i(X)$ is homogeneous of degree $i$. The image by $\theta$ of the inverse of an element $X$ in $\pi$ is given by $\theta(X^{-1})=  \sum_{i=0}^\infty (-1)^i (\theta(X)-1)^i$. Then we compute directly
\begin{align*}
\theta([U,V]) &= \theta(U)\theta(V)\theta(U)^{-1}\theta(V)^{-1} \\ &= 1 + uv-vu +vuv -uvu +vu^2 - uv^2 \\&~~~ + u \theta_2(V) - \theta_2(V)u +\theta_2(U)v - v \theta_2(U) + deg_{\geq4} 
\end{align*}
hence
\begin{align*}
l^\theta([U,V]) &= uv-vu +vuv -uvu +vu^2 - uv^2 \\&~~~ + u \theta_2(V) - \theta_2(V)u +\theta_2(U)v - v \theta_2(U) + deg_{\geq4}\\&= uv -vu + vuv -uvu + \frac{1}{2}(u^2v -v^2u-uv^2+vu^2) \\&~~~ + u l^\theta_2(V) - v l^\theta_2(U) + l^\theta_2(U)v -  l^\theta_2(V)u + deg_{\geq4} \\
&= [u,v] +[u,l_2^\theta(V)]+ [l_2^\theta(U),v] +\frac{1}{2} [u,[u,v]] - \frac{1}{2}[v,[v,u]]+ deg_{\geq4}
\end{align*} which depends only on the expression of $l^\theta$ up to degree $2$. Alternatively, one can use the Baker-Campbell-Hausdorff formula to compute $l^\theta([U,V])$.
\end{proof}

Consequently, we will use for our computations a symplectic expansion given up to degree~$2$. Specifically, we use the truncation of the one given up to degree 4 in \cite[Example 2.19]{masinf}. The fact that it verifies the symplectic condition (up to degree 3) can be checked by hand. The based versions of the curves $\alpha_i$ and $\beta_i$ defining elements of $\pi$ are as shown in Figure $\ref{basispi}$.

\begin{propo}
There is a symplectic expansion $\theta$ of $\pi$, which is given in degree $\leq 2$ by \begin{align*}
l^\theta\left(\alpha_{i}\right)&=a_{i}-\frac{1}{2}\left[a_{i}, b_{i}\right] + deg_{\geq3} \\
l^\theta\left(\beta_{i}\right)&=b_{i}-\frac{1}{2}\left[a_{i}, b_{i}\right]+ deg_{\geq3}.
\end{align*}
\end{propo}

Before giving the description of the elements $\psi$ and $\varphi$ we aim at, we clarify the way we compute in $D_k(H)$, for $1 \leq k \leq 3$. 

\subsection{Derivations of degree 1, 2 and 3}
We recall the description of the spaces $D_k(H)$ in terms of tree-like Jacobi diagrams. More precisely, we consider the spaces of tree-like Jacobi diagram $\mathcal{A}^t(H)$ and rooted tree-like Jacobi diagrams $\mathcal{A}^{t,r}(H)$. A tree is a connected graph that is contractible as a topological space. From now on, by ``a tree", we mean a uni-trivalent tree $T$, possibly rooted, whose trivalent vertices are oriented (the orientation being counterclockwise in all the figures), and whose univalent vertices are colored by elements of $H$. The cardinality of the set of trivalent vertices $v_3(T)$ is the \textit{degree} of the tree $T$. The space $\mathcal{A}_k^t(H)$ (resp. $\mathcal{A}_k^{t,r}(H) $) is the $\mathbb{Z}$-module generated by trees (resp. rooted trees) of degree $k$ subject to some relations: multilinearity of the labels, the \textit{AS relation}, and the \textit{IHX relation} (see  Figure \ref{jacobirelations}). 

\begin{figure}[h]
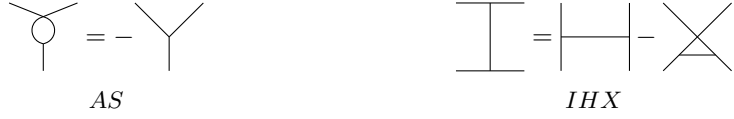

\hspace{15pt}
\vspace{12pt}
\centering
\begin{subfigure}[b]{0.45\textwidth}
\centering
${\tikz[baseline=10pt, scale = 0.6]{
\draw[color=wqwqwq] (0.75,0)-- (0.75,0.6);
\draw[color=wqwqwq] (0.75,0.6) arc (-80:70:0.3) -- (0,1.5);
\draw[color=wqwqwq] (0.75,0.6) arc (-100:-250:0.3)--(1.5,1.5) }} = - {\tikz[baseline=10pt, scale = 0.6]{
\draw[color=wqwqwq] (0.75,0)-- (0.75,0.75);
\draw[color=wqwqwq] (0.75,0.75) -- (0,1.5);
\draw[color=wqwqwq] (0.75,0.75) --(1.5,1.5) }}$

 \caption*{$AS$}
\end{subfigure}
\hfill
\begin{subfigure}[b]{0.45\textwidth}
\centering
${\tikz[baseline=10pt, scale = 0.6]{
\draw [color=wqwqwq] (0,0)-- (1.5,0);
\draw [color=wqwqwq] (0.75,0)-- (0.75,1.5);
\draw [color=wqwqwq] (0,1.5)-- (1.5,1.5);}} = {\tikz[baseline=10pt, scale = 0.6]{
\draw [color=wqwqwq] (0,0)-- (0,1.5);
\draw [color=wqwqwq] (0,0.75)-- (1.5,0.75);
\draw [color=wqwqwq] (1.5,0)-- (1.5,1.5);}} - {\tikz[baseline=10pt, scale = 0.6]{
\draw [color=wqwqwq] (0,0)-- (1.5,1.5);
\draw [color=wqwqwq] (0,1.5)-- (1.5,0);
\draw [color=wqwqwq] (0.35,0.35)-- (1.15,0.35);}}$
\caption*{$IHX$}
\end{subfigure}
\hspace{15pt}
\caption{The $AS$ and $IHX$ relations}
\label{jacobirelations}
\end{figure}
We refer the reader to \cite{levtrees} for further details about what follows. The spaces $\mathcal{A}_k^t(H)$ and $\mathcal{A}_k^{t,r}(H)$ assemble in two graded spaces $\mathcal{A}^t(H)$ and $\mathcal{A}^{t,r}(H) $ endowed respectively with a Lie bracket and a quasi-Lie bracket. For the bracket of $\mathcal{A}^t(H)$, take two trees, and use all the ways to contract external vertices from the first one with the second one using the symplectic form $\omega$. For $\mathcal{A}^{t,r}(H)$, take two trees, glue them along their roots, and add a new root attached to the gluing point (the root must be placed so that the root, the first tree and the second tree are in the clockwise order).

We also define, for $k \geq 1$, maps :
\begin{align*}
 \eta_k : \mathcal{A}^t_k (H) & \longrightarrow D_k(H) \\  T &  \longmapsto  \sum_{x\in v_1(T)} l_x \otimes T^x
\end{align*}
\noindent where $v_1(T)$ is the set of univalent vertices, $l_x$ is the element of $H$ coloring the vertex $x$ and $T^x$ is the rooted tree obtained by setting $x$ to be the root in $T$. A rooted tree can then be read as an element of $\mathcal{L}_{k+1}(H)$ by setting that $\ltritree{*}{$a$}{$b$}$ corresponds to $[b,a]$. These maps, to which we refer as ``the expansion maps", assemble into a graded Lie algebra morphism. The first expansion map $\eta_1$ is an isomorphism, hence any tree with $1$ trivalent vertex in the sequel will represent an element of $D_1(H)$.

The rationalization of the expansion $\eta_{k,\Q} : \mathcal{A}^t_k (H) \otimes \Q \rightarrow D_k(H) \otimes \Q$ is an isomorphism \cite{masb}. Notice that $D_k(H)$ is a free abelian group, and hence is included as a $\Z$-module in $D_k(H) \otimes \Q$. Furthermore, as proved in \cite{levtrees}, any element in $D_2(H) \subset D_2(H) \otimes \Q$ can be obtained by linear combination of expansion of elements of $\mathcal{A}^t_2 (H)$, and expansion of halves of symmetric trees of degree 2, written $a \odot b := \eta_{2,\Q} \bigg(\frac{1}{2} \ltree{$a$}{$b$}{$a$}{$b$}\bigg)$ . Any element in $D_3(H)$ is obtained by summing expansions of elements of $\mathcal{A}^t_3 (H)$ (but $\eta_3$ is not injective). We will consequently refer to such trees as elements of $D_k(H) \subset H \otimes \mathcal{L}_{k+1}(H) \subset H_\Q \otimes \widehat{T_1}$, for $k=2,3$, getting rid of the map $\eta$ that should appear everywhere. We now recall an important formula from \cite{mor}.

\begin{lemma}[Morita]
Let $\gamma$ be a SCC bounding a subsurface $F$ of genus $h$ in $\Sigma$, and let $(u_i,v_i)_{ 1 \leq i \leq h}$ be any symplectic basis of the first homology group of $F$, then we have: \[ \tau_2(T_\gamma) = \sum\limits_{i = 1}^{h} u_i \odot v_i + \sum\limits_{\substack{ i,j = 1\\ i < j}}^{h} \ltree{$u_i$}{$v_i$}{$u_j$}{$v_j$} \in D_2(H). \]
\label{morformula}
\end{lemma}

\label{sectrees}
\subsection{The map $\psi$ as a product of Dehn twists}

In this section, we take $\Sigma:= \Sigma_{2,1}$. Recall that $\kp$ and $\kpp$ denote respectively the subgroups of $\K$ generated by BSCC maps of genus 1 and 2. In \cite{masY3}, following a claim by Morita in $\cite{mor}$, Massuyeau and Meilhan explicited an element of $D_2(H)$ in $3(\tau_2(\kpp))$ which is also in $\tau_2(\kp)$. Indeed it can be checked that, in $D_2(H)$:
\begin{align}
&3(\frac{1}{2}\ltree{$a_1$}{$b_1$}{$a_1$}{$b_1$} + \ltree{$a_1$}{$b_1$}{$a_2$}{$b_2$}+ \frac{1}{2}\ltree{$a_2$}{$b_2$}{$a_2$}{$b_2$})\notag\\  =& 7\cdot \odottree{a_1}{b_1} + 2\cdot\odottree{a_2}{b_2} -\odottree{a_1}{(b_1+b_2)} +\odottree{(b_1+a_2)}{b_2}-\odottree{(a_1+a_2)}{b_1}\notag\\
&- \odottree{(a_1+b_1+a_2)}{b_2} + \odottree{(a_1+a_2+b_2)}{b_1} +\odottree{a_1}{(b_1+a_2+b_2)} \label{calculs1}\\&-\odottree{a_2}{(a_1+b_1+b_2)}+2 \cdot \odottree{a_1}{(b_1+a_2)}+2\cdot\odottree{a_2}{(a_1+b_2)}
-\odottree{(a_1-b_2)}{b_1} \notag\\ &-\odottree{a_1}{(b_1-a_2)} -\odottree{(2a_1+b_2)}{(b_1+a_2)} +\odottree{(a_1+b_1+a_2)}{(a_1+b_1+b_2)}.\notag 
\end{align}

The right-hand side of equation \eqref{calculs1} is in $\tau_2(\kp)$. Indeed, by Lemma \ref{morformula} applied to the boundary of the neighborhood of two simple closed curves with a single intersection point inducing $u,v \in H$ such that $\omega(u,v) = 1$, we get that $u \odot v \in \tau_2(\kp)$. The left-hand side can be obtained as $\tau_2(T_{\gamma_2}^3)$ (where $\gamma_2$ is drawn in Figure \ref{twistfig} and is isotopic to the boundary), hence is in $\tau_2(\kpp)$. We deduce that there is a map $\psi_1$ (not unique), which is a product of Dehn twists on BSCC of genus 1, such that $\psi := T_{\gamma_2}^{-{3}} \psi_1$ belongs to $J_3$. Furthermore, we deduce from \eqref{eqmorita} that $\lambda(\psi) = -\frac{1}{24}(d(\psi)) = 1$, considering that $d(T_{\gamma_2})=8$ and that $d$ is trivial over $\kp$.
\begin{figure}[h]
\centering
\begin{subfigure}[b]{0.32\textwidth}
\includegraphics[scale=0.45]{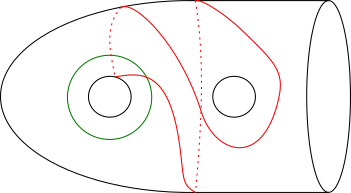}
\caption{Spine of $T_1$}
\end{subfigure}
\begin{subfigure}[b]{0.32\textwidth}
\includegraphics[scale=0.45]{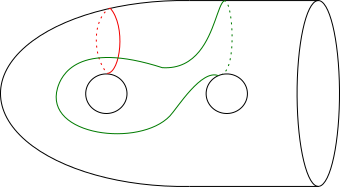}
\caption{Spine of $T_2$}
\end{subfigure}
\begin{subfigure}[b]{0.32\textwidth}
\includegraphics[scale=0.45]{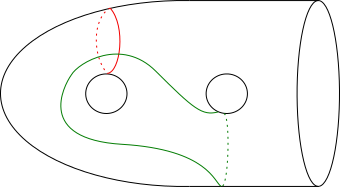}
\caption{Spine of $T_3$}
\end{subfigure}
\begin{subfigure}[b]{0.32\textwidth}
\includegraphics[scale=0.45]{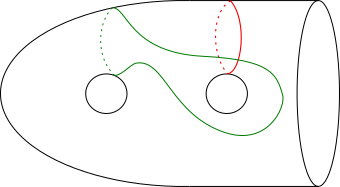}
\caption{Spine of $T_4$}
\end{subfigure}
\begin{subfigure}[b]{0.32\textwidth}
\includegraphics[scale=0.45]{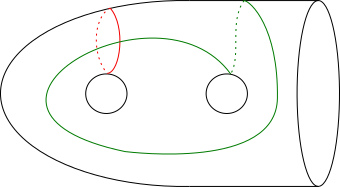}
\caption{Spine of $T_5$}
\end{subfigure}
\begin{subfigure}[b]{0.32\textwidth}
\includegraphics[scale=0.45]{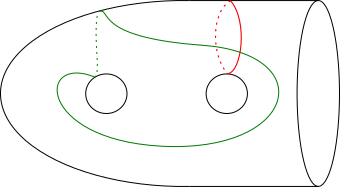}
\caption{Spine of $T_6$}
\end{subfigure}
\begin{subfigure}[b]{0.32\textwidth}
\includegraphics[scale=0.45]{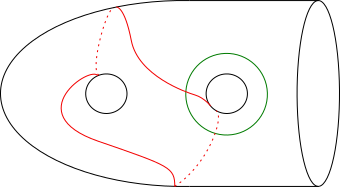}
\caption{Spine of $T_7$}
\end{subfigure}
\begin{subfigure}[b]{0.32\textwidth}
\includegraphics[scale=0.45]{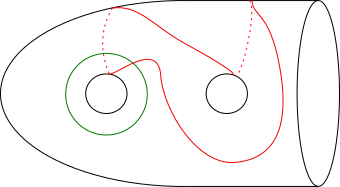}
\caption{Spine of $T_8$}
\end{subfigure}
\begin{subfigure}[b]{0.32\textwidth}
\includegraphics[scale=0.45]{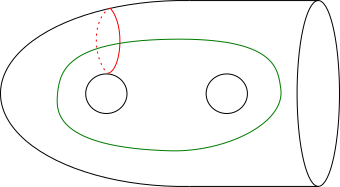}
\caption{Spine of $T_9$}
\end{subfigure}
\begin{subfigure}[b]{0.32\textwidth}
\includegraphics[scale=0.45]{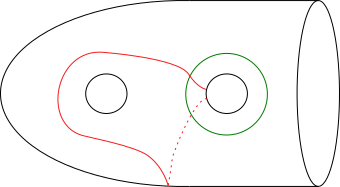}
\caption{Spine of $T_{10}$}
\end{subfigure}
\begin{subfigure}[b]{0.32\textwidth}
\includegraphics[scale=0.45]{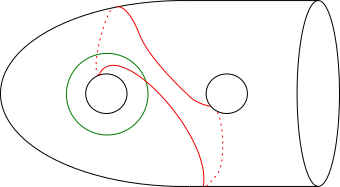}
\caption{Spine of $T_{11}$}
\end{subfigure}
\begin{subfigure}[b]{0.32\textwidth}
\includegraphics[scale=0.45]{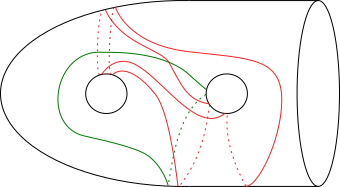}
\caption{Spine of $T_{12}$}
\end{subfigure}
\begin{subfigure}[b]{0.32\textwidth}
\includegraphics[scale=0.45]{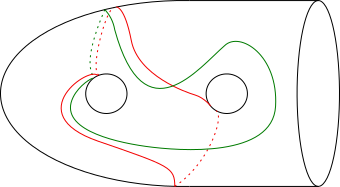}
\caption{Spine of $T_{13}$}
\end{subfigure}
\begin{subfigure}[b]{0.32\textwidth}
\includegraphics[scale=0.45]{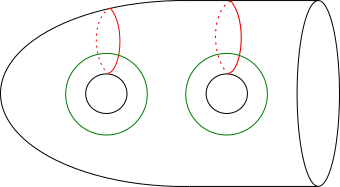}
\caption{Spines of $T_{s1}$ and $T_{s2}$}
\end{subfigure}
\begin{subfigure}[b]{0.32\textwidth}
\includegraphics[scale=0.45]{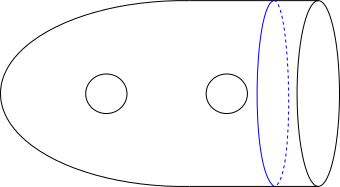}
\caption{The curve $\gamma_2$}
\end{subfigure}

\caption{Twists involved in the definition of $\psi$}
\label{twistfig}
\end{figure}

We now exhibit such a $\psi_1 \in \kp$, thus obtaining an application $\psi \in J_3$ whose Casson invariant is 1. A \emph{spine} of the Dehn twist around a BSCC $\gamma$ of genus 1 designates two simple closed curves (with geometric intersection equal to $1$) such that the boundary of a regular neighborghood of the union of the two curves is isotopic to $\gamma$. The map $\psi_1$ is defined as the product of some twists, the spines of which are drawn in Figure \ref{twistfig}: $$\psi_1 := T_1^{-1}\,T_2^{-1}\,T_3^{2}\,T_4^{2}\,T_5\,T_6^{-1}\,T_7^{-1}\,T_8\,T_9^{-1}\,T_{10}\,T_{11}^{-1}\,T_{12}^{-1}\,T_{13}\,T_{s1}^{7}\,T_{s2}^{2}$$ and we set $\psi := T_{\gamma_2}^{-3}\,\psi_1$. This definition of $\psi$ is rather complicated, involving 10 left Dehn twists and 17 right Dehn twists. Counting algebraically (positively for right Dehn twists and negatively for left Dehn twists), we needed $-3$ twists of genus $2$ and 10 twists of genus $1$ to create our element $\psi$. In fact, one needs so many twists to get $\psi \in J_3$ such that $\lambda(\psi)= 1$, as we shall explain in subsection \ref{submin}.
\label{themappsi}

\subsection{The map $\varphi$}
This section is dedicated to the proof of Theorem \ref{thA}. As $\lambda$ is a homomorphism on $J_4$, we ``simply'' need to find an element $\varphi \in J_4$ with $\lambda(\varphi) =1$. We build $\varphi$ by ``perturbing" the map $\psi$ that has been defined in the previous section. By using the Kawazumi-Kuno formula, we can compute the image of $\psi$ by $\tau_3$ (note here that we actually compute $\tau_{3, \Q}(\psi)$ and not $\tau_3(\psi)$, but the arrow from $D_3(H)$ to its rationalization is injective, hence we indeed get the value of $\tau_3(\psi)$). Indeed, by construction, $\psi$ belongs to $J_3$, hence $\tau_3^\theta(\psi)= \tau_3(\psi)$. Let us write formally $\psi = \Pi_{i=0}^{27} U_i$, where $U_i$ is either a right or left Dehn twist. Then Corollary \ref{KK3} gives $$\tau_3(\psi) = \sum \limits_{i = 0}^{27} \tau_3^\theta(U_i) = \sum \limits_{i = 0}^{27} L_5^\theta(U_i).$$ Here we use the fact that $\tau_2^{\theta} \oplus \tau_3^{\theta}$ is a homomorphism on $J_2$. For the proof of this fact, see \cite{masinf} where it is proven that, for any $k \geq 1$, $\bigoplus \limits_{i \in [k,2k[} \tau_i^\theta$ corresponds to the $k$-th Morita homomorphism on $J_k$ (which is a certain refinement of $\tau_k$).

After implementing this formula in a SageMath computer program (see Appendix \ref{appA}), we get
\begin{align*}
\tau_3(\psi) =& -\lfivetree{$a_2$}{$a_1$}{$a_1$}{$b_1$}{$a_1$}-\lfivetree{$a_2$}{$b_1$}{$a_1$}{$a_2$}{$a_1$}-\lfivetree{$b_2$}{$a_1$}{$a_1$}{$b_1$}{$a_1$}-\lfivetree{$b_2$}{$b_1$}{$a_1$}{$b_1$}{$a_1$}+\lfivetree{$b_2$}{$a_2$}{$a_1$}{$b_1$}{$a_1$}+\lfivetree{$b_2$}{$a_2$}{$a_1$}{$a_2$}{$a_1$}\\&+\lfivetree{$b_2$}{$a_2$}{$a_1$}{$b_2$}{$a_1$}+\lfivetree{$b_2$}{$a_2$}{$b_1$}{$b_2$}{$a_1$}+3*\lfivetree{$b_2$}{$a_2$}{$a_2$}{$b_1$}{$a_1$}+\lfivetree{$b_2$}{$a_2$}{$a_2$}{$a_2$}{$a_1$}+\lfivetree{$b_2$}{$a_2$}{$b_2$}{$b_1$}{$a_1$}\\&-\lfivetree{$b_2$}{$a_1$}{$a_2$}{$b_1$}{$a_1$}+\lfivetree{$b_2$}{$b_1$}{$a_2$}{$b_1$}{$a_1$}+\lfivetree{$b_2$}{$a_2$}{$b_2$}{$a_2$}{$a_1$}-\lfivetree{$b_2$}{$a_2$}{$b_2$}{$a_2$}{$b_1$},
\end{align*} or, in a more compact form:

$\tau_3(\psi) = \bfivetree{$b_1+a_2$}{$a_1$}{$a_1+a_2+b_2$}{$a_2$}{$a_1+b_2$} + \bfivetree{$a_2-a_1$}{$b_2$}{$a_1+b_1$}{$a_1$}{$b_1+b_2$}-\bfivetree{$a_2-a_1$}{$b_1$}{$b_2$}{$a_2$}{$b_1+b_2$}+\bfivetree{$b_2$}{$a_2$}{$2a_2-2a_1+b_2$}{$b_1$}{$a_1$}$.

\noindent Next we rewrite $\tau_3(\psi)$ as a sum of brackets:

\begin{align}
\tau_3(\psi) &= \bracket{3\ltritree{$a_1$}{$b_1$}{$a_2$} +\ltritree{$b_2$}{$a_2$}{$a_1$} +\ltritree{$a_1$}{$b_1$}{$b_2$}}{\ltree{$a_1$}{$b_1$}{$a_2$}{$b_2$}} \notag\\
  &  ~~~+\bracket{\ltritree{$b_1$}{$a_1$}{$a_2$-$b_2$}}{\ltree{$a_1$}{$a_2$}{$b_2$}{$a_1$}}+\bracket{\ltritree{$a_2$}{$b_2$}{$a_1$}}{\ltree{$a_2$}{$b_1$}{$a_1$}{$a_2$}} \label{eqbrack}\\
  &~~~+\bracket{\ltritree{$a_1$}{$b_1$}{$a_2$}}{\bracket{\ltritree{$a_1$}{$b_1$}{$b_2$}}{\ltritree{$a_1$-$b_1$}{$a_2$}{$b_2$}}} \notag\\
  &~~~+\bracket{\ltritree{$b_1$}{$a_2$}{$b_2$}}{\bracket{\ltritree{$a_1$}{$b_2$}{$a_2$}}{\ltritree{$a_1$}{$b_1$}{$b_2$}}}.\notag
 \end{align}
 
\noindent Recall that the trees appearing in the above formulas  define derivations through the expansion map $\eta$, which is a Lie homomorphism. The trees of degree $2$ inside the above brackets actually are in $\tau_2(\K)$. To see this, one can simply observe that they are in the kernel of the ``antisymmetric'' trace that has been defined in \cite{faes} to characterize $\tau_2(\K)$. We can also show this directly:

\begin{lemma}
The trees $\ltree{$a_1$}{$b_1$}{$a_2$}{$b_2$}$ , $\ltree{$a_1$}{$a_2$}{$b_2$}{$a_1$}$ and $\ltree{$a_2$}{$b_1$}{$a_1$}{$a_2$}$ are in $\tau_2(\K)$.
\label{lemtrees}
\end{lemma}

\begin{proof}
By the formula in Lemma \ref{morformula} we deduce that $$\ltree{$a_1$}{$b_1$}{$a_2$}{$b_2$} = \tau_2(T_{\gamma_2}T_{s1}^{-1}T_{s2}^{-1}) $$
\noindent It is easily seen that the two other trees are in the same orbit under the action of $\sph$. Hence, it is enough to show that $\ltree{$a_2$}{$b_1$}{$a_1$}{$a_2$}$ belongs to $\tau_2(\K)$. For this, we note that $$ \ltree{$a_2$}{$b_1$}{$a_1$}{$a_2$} = a_1 \odot b_1 - a_1 \odot (b_1 + a_2) - (a_1 +a_2) \odot b_1 +  (a_1 +a_2) \odot (b_1+a_2).$$
Any element of the form $u \odot v$ with $\omega(u,v)=1$ being in $\tau_2(\K)$, we conclude that $ \ltree{$a_2$}{$b_1$}{$a_1$}{$a_2$}$ is indeed an element of $\tau_2(\K)$.
\end{proof}

By Lemma \ref{lemtrees}, we get that $\tau_3(\psi) \in \tau_3([\K, \I])$. Indeed, Johnson has shown that $\tau_1$ is onto $D_1(H)$ \cite{joh80}, hence any tree of degree $1$ is the image by $\tau_1$ of an element in the Torelli group. Pick any element $\psi_2 \in [\K, \I]$, such that $\tau_3(\psi) + \tau_3(\psi_2) = 0$, and define $\varphi := \psi \psi_2 \in J_3$, so that $\varphi$ belongs to $\operatorname{Ker}(\tau_3) = J_4$. Since $d : \K \rightarrow \Z$ is invariant under the conjugacy action of $\M$ (see \cite{mor91}), we have $d([\K, \M]) = 0$, so that $d(\psi_2) =0$. We get $$\lambda(\varphi) = -\frac{1}{24}d(\varphi) = -\frac{1}{24}d(\psi) = 1,$$ which concludes the proof of Theorem \ref{thA}.

\begin{rem}
Recall that $J_4$ contains $[J_1,J_3]= [\I,J_3]$ and $[J_2,J_2]= [\K,\K]$ as subgroups. Our map $\varphi$ seems to be the first explicit example of an element of $J_4$ which is not in $[\mathcal{I}, J_3]$ nor in $[\K, \K]$ since it is not in  $[\K, \M]$ (and none of its powers is, actually). By ``explicit'', we mean that we are able to decompose the map $\varphi = \psi\psi_2$ as a product of Dehn twists. Specifically, one can build $\psi_2$ with $22$ twists. To do so, we use the decomposition of $\tau_3(\psi)= -\tau_3(\psi_2)$ in brackets given by equation \eqref{eqbrack} and we use the following fact: if $s,t \in \M$ and $t$ is a product of $k \geq 1$ twists, then $[s,t]$ can be written as a product of $2k$ twists. As $\psi$ is a product of $27$ twists, we deduce that $\varphi$ can be given as a product of $49$ twists. Nozaki, Sato and Suzuki obtained in \cite[Corollary 1.6]{nss}, in the case of a closed surface and for $g \geq 6$, the existence of an element of $J_4$, which was not in $[\K, \K]$: this element has a power in $[\K, \K]$, though.
\end{rem}

\begin{rem}
Note also that $\varphi$ does not belong to $[\I,\I]$ neither. Indeed the fact that $\lambda_j(\phi)=1$ implies that the Birman-Craggs homomorphism (see \cite{bircra,johquad}) does not vanish on $\varphi$.
\end{rem}
\label{subseccomp}
\subsection{Complexity of the computation}

One could wonder whether it is possible to build elements similar to $\psi \in J_3$ and $\varphi \in J_4$ using fewer twists. In this subsection we explain why so many Dehn twists have been necessary (see Proposition~\ref{complexity} below). 

Recall that $\kp$ and $\kp'$ are the subgroups of $\K$ generated, respectively, by twists around BSCC of genus 1 and 2. We will first give a description of the quotient $\tau_2(\K)/\tau_2(\kp)$. Here we suppose that the genus $g$ of $\Sigma$ is at least $2$, so that, by a result of Johnson \cite{joh2}, the group $\K$ is normally generated by any BSCC map of genus $1$ and any BSCC map of genus~$2$. This implies, by Lemma \ref{morformula}, that $\tau_2(\K)$ is generated as a $\sph$-module by $\frac{1}{2} \ltree{$a_1$}{$b_1$}{$a_1$}{$b_1$}$ and $\ltree{$a_1$}{$b_1$}{$a_2$}{$b_2$}$. It is also clear that $\tau_2(\kp)$ is the $\sph$-submodule of $D_2(H)$ generated by $\frac{1}{2} \ltree{$a_1$}{$b_1$}{$a_1$}{$b_1$}$. Equation \eqref{calculs1} implies that $3\ltree{$a_1$}{$b_1$}{$a_2$}{$b_2$} \in \tau_2(\kp)$, hence we get the following:

\begin{coroll}
For any $g\geq 2$, the quotient $\tau_2(\mathcal{K})/ \tau_2(\mathcal{K'})$ is a $\mathbb{Z}_3$-module generated as a $\sph$-module by the class of $\ltree{$a_1$}{$b_1$}{$a_2$}{$b_2$}$.
\end{coroll}

\noindent From this, we get a precise description of the quotient $\tau_2(\K)/\tau_2(\kp)$.

\begin{propo}
For any $g \geq 2$, we have the following $\sph$-module isomorphisms: $$ \frac{D_2(H)}{\tau_2(\kp)}\otimes \mathbb{Z}_3 \simeq \frac{\tau_2(\mathcal{K})}{\tau_2(\kp)} \simeq \Lambda^4(H/3H).$$
\label{tau2kp}
\end{propo}

\begin{proof}
We denote by $Q$ the quotient $\frac{D_2(H)}{\tau_2(\kp)}\otimes \mathbb{Z}_3$.
The first isomorphism is a direct consequence of the fact that $\tau_2(\mathcal{K})$ is of order 2 in $D_2(H)$, and $2$ is invertible in $\mathbb{Z}_3$. Indeed, the map from $Q$ to $\frac{\tau_2(\mathcal{K})}{\tau_2(\kp)}$ sending an element $x$ to the class of $4x$ in $\frac{\tau_2(\mathcal{K})}{\tau_2(\kp)}$ is well-defined (because $D_2(H)/\tau_2(\K)$ is a 2-torsion module \cite[Prop 1.2]{mor}), and has for inverse the map induced by the inclusion of $\tau_2(\K)$ in $D_2(H)$.

We then use the presentation of $D_2(H)$ from \cite[Prop. 2.1]{faes} to define a homomorphism $\kappa : \frac{D_2(H)}{\tau_2(\kp)}\otimes \mathbb{Z}_3 \rightarrow \Lambda^4(H/3H)$, by putting $\kappa(a \odot b) := 0$, for any $a,b \in H$, and $\kappa(\ltree{$a$}{$b$}{$c$}{$d$}) := a \wedge b \wedge c \wedge d \in \Lambda^4(H/3H)$ for any $a,b,c,d \in H$. It is straightforward to check that all the relations in $D_2(H)$, are sent to $0$, except maybe the IHX relation:
\begin{align*}
\kappa(\ltree{$a$}{$b$}{$c$}{$d$} - \ltree{$a$}{$c$}{$b$}{$d$} - \ltree{$a$}{$d$}{$c$}{$b$})&= a\wedge b\wedge c \wedge d  - a \wedge c \wedge b \wedge d - a \wedge d \wedge c \wedge b \\
&= 3 (a \wedge b \wedge c \wedge d) \\
&= 0 \in \Lambda^4(H/3H).
\end{align*}
The map $\kappa$ is then a well-defined $\operatorname{Sp}(H)$-equivariant homomorphism.

Reciprocally, we show that the map $\nu$, sending $a \wedge b \wedge c \wedge d$ to the class of $\ltree{$a$}{$b$}{$c$}{$d$}$ in $Q$ is also well-defined. It will follow that $\kappa$ is an isomorphism with $\kappa^{-1} = \nu$. To prove that $\nu$ is well-defined, we need to show that the relation $a \wedge b \wedge c \wedge d + a \wedge c \wedge b \wedge d = 0$ is sent to $0 \in Q$ (the other relations in $\Lambda^4(H/3H)$ are easily verified). Hence we want to show that $\ltree{$a$}{$b$}{$c$}{$d$} + \ltree{$a$}{$c$}{$b$}{$d$}$ is equal to $0$ in $Q$.

We first show that for any $a,b$ in $H$, $a \odot b$ is equal to $0$ in $Q$. The equality $a \odot b = b \odot a$, for any $a, b \in H$ allows us to suppose that we always have $\omega(a,b) \geq 0$. We proceed by induction on $\omega(a,b) \in \mathbb{N}$. If $\omega(a,b) = 1$, we know that there exists a subsurface $F$ of genus $1$ in $\Sigma$, with one boudary component $\gamma$, such that there are two simple closed curves $\alpha$ and $\beta$, inducing $a$ and $b$ in homology, intersecting once and inducing a symplectic basis of $F$. We then have by Lemma \ref{morformula} that $\tau_2(T_\gamma) = a \odot b$, hence $a \odot b = 0 \in Q$ . If $c \in H$ is such that  $\omega(a,c)= 0$, we have $\omega(a,b+c) = \omega(a,b-c)=1$, and:
\begin{align*}
 a \odot (b+c) - a \odot b - a \odot c &= \ltree{$a$}{$b$}{$a$}{$c$} \\ 
 a \odot (b-c) - a \odot b - a \odot c &= -\ltree{$a$}{$b$}{$a$}{$c$}
 \end{align*}
which implies that $2 a \odot c$, and hence $a \odot c$, is trivial in $Q$. Now, if $\omega(a,b)=k +1$ for some $k \geq 1$, either $a$ is not primitive in $H$ (and we can reduce the value of $k$), either there exists $b_2 \in H$ such that $\omega(a,b_2)=1$, and we can write $b = b_1 + b_2$ with $\omega(a,b_1) = k$. Thus we have that $a \odot (b_1 +b_2) + a \odot (b_1 -b_2) = 2(a \odot b_1) + 2(a \odot b_2) \in Q$. By using the induction hypothesis for $k, k-1$ and $1$, we conclude that $a \odot b = a \odot (b_1 +b_2)$ is trivial in $Q$. 

Furthermore, for any $a, b,c,d \in H$, we have the following, where all the terms of the left hand-sides are trivial in~$Q$:
\begin{align*}
(a +c) \odot b - a \odot b - c \odot b &= \ltree{$a$}{$b$}{$c$}{$b$} \\
\ltree{$a$}{$b+d$}{$c$}{$b+d$} - \ltree{$a$}{$b$}{$c$}{$b$}  - \ltree{$a$}{$d$}{$c$}{$d$} &= \ltree{$a$}{$b$}{$c$}{$d$} + \ltree{$a$}{$d$}{$c$}{$b$} \\
&= 2\ltree{$a$}{$b$}{$c$}{$d$} -\ltree{$a$}{$c$}{$b$}{$d$}.
\end{align*}
We conclude, because we work modulo $3$, that $\ltree{$a$}{$b$}{$c$}{$d$} +\ltree{$a$}{$c$}{$b$}{$d$}$ is trivial in $Q$.
\end{proof}

%%If we denote $pr$ the projection from $\tau_2(\mathcal{K})$ to the quotient $Q \simeq \tau_2(\K)/\tau_2(\kp)$, then we have $\kp \subset \operatorname{Ker}(pr \circ \tau_2)$. 
The characterization given by Proposition \ref{tau2kp} might be helpful to build other elements $\psi \in J_3$ such that $d(\psi) =1$, making use of the fact that $d$ is trivial on $\kp$. Indeed, for any product of left and right Dehn twists of genus $2$, such that the algebraic number of twists is $-3$, and whose image by $\tau_2$ is in $\tau_2(\kp)$, we can always multiply it by an element of $\kp$ such that the product is in $J_3$ and has Casson invariant equal to $1$. 

Nevertheless, this will always involve a lot of Dehn twists, as we now explain. This is related to the properties of $d$, the core the Casson invariant, and another invariant $d'$ defined by Morita in $\cite{mor91}$. As we recalled, the map $d$ is the homomorphism from $\K$ to $\Z$ sending a twist of genus $h$ to $4h(h-1)$. Similarly, $d'$ can be defined as the homomorphism from $\K$ to $\Z$ sending a twist of genus $h$ to $h(2h+1)$. The particularity of the map $d'$ is that it factors through $\tau_2$:

\begin{equation}
d' := \bar{d'} \circ \tau_2
\label{eqdp}
\end{equation} 
where $\bar{d'}$ is defined on $D_2(H)$ by

\begin{align*}
\bar{d'}(a \odot b) &:= 3 \omega(a,b)^2 \\
\bar{d'}(\ltree{$a$}{$b$}{$c$}{$d$}) &:= 4\omega(a,b)\omega(c,d) - 2\omega(a,d)\omega(b,c) + 2 \omega(a,c)\omega(b,d).
\end{align*}

\noindent Morita \cite[Th. 5.4]{mor91} showed that the $\M$-equivariant homomorphisms from $\K$ to $\mathbb{Z}$ are rational linear combinations of $d$ and $d'$ (which are linearly independent): thus $H^{1}(\mathcal{K;\mathbb{Z}})^{\M}$ is free abelian of rank two and generated over the rational by $d$ and $d'$. In fact the following lemma can be obtained by direct computation.

\begin{lemma}
If $T_1$ is a twist of genus $1$ and $T_2$ is a twist of genus $2$, then:
\begin{align*}
\frac{d}{8}(T_2) &= \frac{4d' -5d}{12}(T_1) = 1 \\
\frac{d}{8}(T_1) &= \frac{4d' -5d}{12}(T_2) = 0.
\end{align*}
\label{lemddp}
\end{lemma}
\noindent Therefore we obtain that the abelian group $H^{1}(\mathcal{K;\mathbb{Z}})^{\M}$ is freely generated by $\frac{4d'-5d}{12}$ and $\frac{d}{8}$. Furthermore, we have the following.

\begin{propo}
For any $g \geq 2$, $\mathcal{K}/[\mathcal{K}, \mathcal{M}]$ is free abelian of rank $2$ and canonically isomomorphic to its dual $H^{1}(\mathcal{K;\mathbb{Z}})^{\M}.$
\label{propKM}
\end{propo}

\begin{proof}
Choose two twists of genus $1$ and $2$, denoted respectively $T_1$ and $T_2$. Then, the map from $H^{1}(\mathcal{K;\mathbb{Z}})^{\M}$ to $\K/[\K,\M]$ defined by sending $f$ to $T_1^{f(T_1)}T_2^{f(T_2)}$ is surjective and canonical, as any two Dehn twists of same genus are always conjugated to each other. It is injective by Lemma \ref{lemddp}.
\end{proof}

\begin{rem}
Proposition \ref{propKM} implies that to compute the algebraic number of Dehn twists of genus 1 and 2 involved in an element of $\K$, it is enough to know its values under $\tau_2$ and $d$. Also, for any $\phi \in \operatorname{Ker}(d)$, we see that there is a $k \in \mathbb{Z}$ such that $d'(\phi) = 3k$. Then we get that for any BSCC $\gamma$ of genus 1, $\phi (T_\gamma)^{-k}$ is in $\operatorname{Ker}(d) \cap \operatorname{Ker}(d') = [\K,\M]$. This proves that in genus $g \geq 2$, $\operatorname{Ker}(d) = [\K,\M]\kp$.
\end{rem}

To get an element $\psi \in J_3$ such that $\lambda(\psi) = 1$, we need to have $d(\psi)= -24$, and $d'(\psi)=0$ (as $d'$ factorizes by $\tau_2$). This implies that $\frac{4d' -5d}{12}(\psi) = 10$ and $\frac{d}{8}(\psi)= -3$. Thus, by the previous computations, the algebraic numbers of Dehn twists of genus $1$ and $2$ in $\psi$ are respectively $10$ and $-3$. Thus we obtain the following:

\begin{propo}
In genus $g \geq 2$, any element in $J_3$ whose Casson invariant is $1$ is the composition of BSCC maps of genus $1$ and $2$ such that the algebraic number of BSCC maps of genus $1$ is $10$ and the algebraic number of BSCC maps of genus $2$ is $-3$.
\label{complexity}
\end{propo}

\label{submin}

\section{Triviality of the $J_4$-equivalence}
In this section, we prove Theorem \ref{thB}. We have shown in Section \ref{sec2} that there exists an element $\psi \in J_4$ whose Casson invariant is equal to $1$. Following \cite[Rem. 6.4]{masY3}, we explain how this implies the triviality of the $J_4$-equivalence relation on the set of homology 3-spheres.

We first recall the definition of the $Y_k$-equivalence and the $J_k$-equivalence relations, and refer to \cite[Section 2]{masY3} for more details. For a compact oriented 3-manifold $M$, a submanifold $S \subset \operatorname{int}(M)$ homeomorphic to $\Sigma_{g,1}$ for some $g \geq 1$, and any $\varphi$ in the mapping class group of $S$, we define the 3-manifold $M_{S, \varphi}$ by removing from $M$ a neighborhood $S \times [-1,1]$ of $S$, and regluing it twisting by the map $\varphi$:

 \[M_{S,\varphi} := (M \backslash \operatorname{int}(S \times[-1,1])) \cup_{\overline{\varphi}}(S \times[-1,1])  \] \noindent where $\overline{\varphi}$ is the map from $\partial (S \times [-1,1])$ to $M$ defined by $(Id \times (-1)) \cup (  \partial S \times Id) \cup  (\varphi \times 1)$.

Whenever the map $\varphi$ is in the Torelli group $J_1$ of $S$, we call the move from $M$ to $M_{S,\varphi}$ a \emph{Torelli surgery}. Torelli surgeries preserve the set $\mathcal{S}(3)$ of homeomorphism classes of homology 3-spheres.

\begin{definition}
The $Y_k$-equivalence and \textit{$J_k$-equivalence relations} are defined on $\mathcal{S}(3)$ as follows: \[ M \stackrel{Y_{k}}{\sim} M^{\prime} \Leftrightarrow \exists S \subset M, ~  \exists \varphi  \in \Gamma_{k} \mathcal{I}(S) \textit{ such that } M' \cong M_{S,\varphi}, \]
\[ M \stackrel{J_{k}}{\sim} M^{\prime} \Leftrightarrow \exists S \subset M, ~  \exists \varphi  \in J_k(S) \textit{ such that } M' \cong M_{S,\varphi}. \]
\end{definition}

These relations are equivalence relations, and $Y_k$-equivalence implies $J_k$-equivalence. It is known that $J_3$-equivalence is trivial among homology 3-spheres \cite{pit}. We now improve this result. We first sketch the proof of the following fact, announced by Habiro \cite{habclasper}:

\begin{theorem}[Habiro]
For any $M, M' \in \mathcal{S}(3)$, the following statements are equivalent:

\begin{align*}
(1) ~~ &  M \stackrel{Y_{3}}{\sim} M^{\prime} \\
(2)~~& M \stackrel{Y_{4}}{\sim} M^{\prime} \\
(3)~~&\lambda(M) = \lambda(M') \in \mathbb{Z}.
\end{align*}
\label{Y4lemma}
\end{theorem} 

\begin{proof}[Sketch of proof]
Habiro \cite[Section 8]{habclasper} studied the \emph{$Y$-filtration} on the monoid of homology cylinders (we refer to \cite[Sections 5 and 6]{mhsurvey} for a survey). A similar study was made by Goussarov \cite{gous}, with a	different vocabulary. Here, we need only to use the results of Habiro in the genus $0$ case, corresponding to the monoid of homology $3$-spheres (where the multiplication is the connected sum operation). In this case $\mathcal{S}(3)$ is filtered by $(Y_k(\mathcal{S}(3)))_{k\geq 1}$, where $Y_k(\mathcal{S}(3))$ is the submonoid consisting of homology 3-spheres $Y_k$-equivalent to $S^3$. For any $k \geq 1$, the quotient of $\mathcal{S}(3)$ by the $Y_k$-equivalence is a group and the quotient of $Y_k(\mathcal{S}(3))$ by the $Y_{k+1}$-equivalence is an abelian group. We consider the associated graded space $$\operatorname{Gr}^Y(\mathcal{S}(3)):= \bigoplus_{k \geq 1} Y_k(\mathcal{S}(3))/ Y_{k+1}.$$ By using the techniques of clasper surgery, Habiro was able to define (in degree greater than $2$) a surjective graded map from a certain space of Jacobi diagrams, namely trivalent graphs subject to the $AS$ and $IHX$ relations (see Figure \ref{jacobirelations}), to $\operatorname{Gr}^Y(\mathcal{S}(3))$. In degree $2$, there is only one (up to scalar) such Jacobi diagram, and in degree $3$, there are none. This implies that $Y_2(\mathcal{S}(3))/ Y_{3} \simeq \Z$ (the isomorphism being given by the Casson invariant), and $Y_3(\mathcal{S}(3))/ Y_{4} = 0$. We deduce that two homology $3$-spheres are $Y_4$-equivalent if and only if they are $Y_3$-equivalent, i.e. if and only if they have the same Casson invariant (see \cite{habclasper} or \cite{masY3}). 
\end{proof}

\begin{proof}[Proof of Theorem \ref{thB}]
Let $S$ be such that the standard Heegaard surface of genus $2$ in $S^3$ is obtained from $S$ by capping it with a disk. Let us pick $\varphi \in J_4(S)$ such that $\lambda(\varphi)=1$: the existence of such an element has been proved in Section \ref{sec2}. We have by definition that $P:=S^3_{S, \varphi}$ is a homology 3-sphere whose Casson invariant is equal to $1$. Furthermore, the homology 3-sphere $P$ is by construction $J_4$-equivalent to $S^3$. Let $M$ and $M'$ be two homology 3-spheres, and assume that $\lambda(M) \leq \lambda(M')$. Then, the additivity of the Casson invariant implies that $\lambda(M \# P^{(\lambda(M') - \lambda(M))}) = \lambda(M')$, hence

\[ M \cong M \# ({S^3})^{(\lambda(M') - \lambda(M))} \stackrel{J_{4}}{\sim} M \# P^{(\lambda(M') - \lambda(M))} \stackrel{Y_{4}}{\sim} M'\] where the last equality follows from Theorem \ref{Y4lemma}. By transitivity, we get that $M$ is $J_4$-equivalent to $M'$.
\end{proof}
\begin{rem}
Knowing the values of $\lambda$ on $J_5$ would \textit{a priori} not be enough to discuss the $J_5$-equivalence, as the $Y_5$-equivalence is related to higher-order finite-type invariants.
\end{rem}
\label{sec3}

\begin{appendices}
\titlelabel{\normalsize APPENDIX \thetitle.  }
\titleformat*{\section}{\Large\bfseries\scshape}
\section{SageMath computer program}

In this appendix, we provide the code that we used to verify that the map $\psi$ constructed in Section \ref{themappsi} belongs to $J_3$, compute $\tau_3(\psi)$ and write it as a linear combination of degree 3 tree-like Jacobi diagrams.

We divide the program in two parts. The first part of the program is functional for any genus $g \geq 1$ and allows one to compute the maps $L^\theta_k$ on commutators, for $k \in \{4,\dots,N\}$, provided we are given a symplectic expansion $\theta$ up to degree $N-2$. Here we take $N:= 5$ and the expansion used in the program is a truncation of the one given in \cite{masinf}. The entries of the functions $L^\theta_k$ in the program are strings describing elements in $\pi$ (for example $\alpha_1\beta_2^{-1}$ is encoded as `a1+b2-'). The output of the function $L^\theta_k$ in the code is actually an element of degree $k$ in the free algebra on $2g$ generators, instead of an element of the tensor algebra $\widehat{T}$. The second part of the code is specific to the genus $g:=2$, and allows us to compute $\tau_3(\psi)$ by entering \emph{barcodes} associated to the spines in Figure \ref{twistfig}. A barcode is a list of non-zero integers between $g$ and $-g$, which corresponds to a word in $\pi$ through the correspondences $\pm (2i-1) \sim \alpha_i^{\pm 1}$ and $\pm 2i \sim \beta_i^{\pm 1}$. Such a barcode is then transformed in a string encoding an element of $\pi$ and we use the function $L_k^\theta$ to compute $\tau_3(\psi)$. In the end,  the program compares $\tau_3(\psi)$ to the expansion of trees claimed in Section \ref{subseccomp}.

We explain on an example how to construct a barcode encoding a curve in the surface. We consider the boundary of a neighborhood of the spine of the twist $T_7$ in Figure \ref{twistfig}, orient it in any way, and denote it $\delta_7$. We lift $\delta_7$ to an element $y_7$ of the fundamental group in an arbitrary way, by linking it to the base point with an arc. 
\begin{figure}[h]
	\centering
	\includegraphics[scale=1.2]{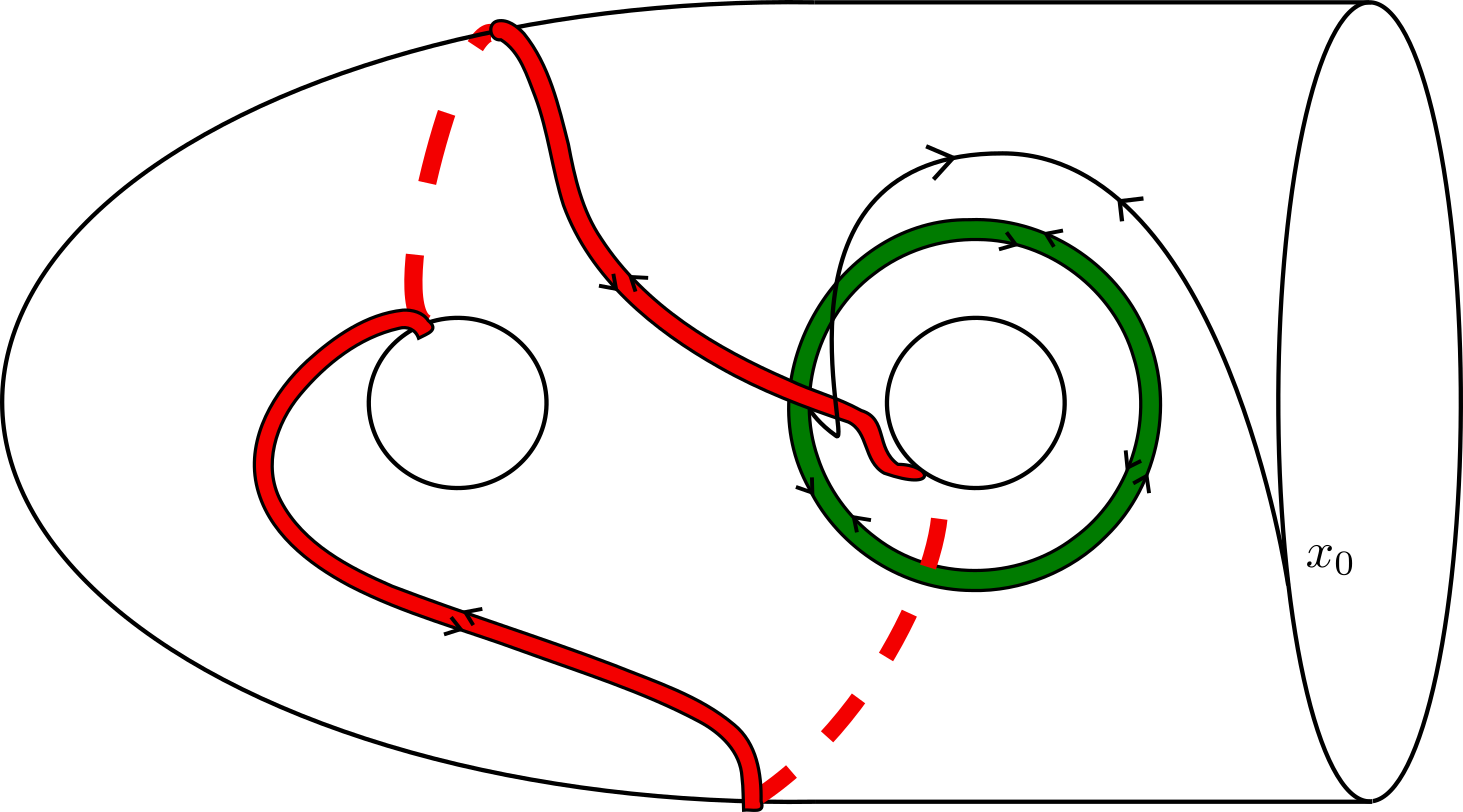}
	\caption{The element $y_7 \in \pi$, based at $x_0$}
	\label{surface2}
\end{figure}
Precisely, the based curve $y_7$ is obtained by first travelling from $x_0$ to the boundary of the spine along this arc (drawn with two opposite arrows in Figure \ref{surface2}), then going along $\delta_7$, and finally travelling back along the arc from the spine to $x_0$. We can now express the curve defined by the spine as a word in the fundamental group, using the basis of $\pi$ chosen in Figure \ref{basispi}. Respecting this procedure, one will obtain a commutator $[U,V] \in \pi$ where $U$ and $V$ are elements corresponding to the curves of the spine, lifted to elements of $\pi$ using the arc $\gamma$. In our case, we get that $y_7 = [\alpha_2^{-1}\beta_2[\alpha_1, \beta_1^{-1}]\beta_1^{-1}\alpha_1^{-1}\beta_2^{-1},\beta_2] \in \pi$, where the brackets refer to the commutator bracket in the free group $\pi$. We now use the correspondency described above and we get that the barcode corresponding to $y_7$ is $[-3, 4, 1, -2, -1, 2, -2, -1, -4, 4, 4, 1, 2, -2, 1, 2, -1, -4, 3, -4]$. Of course, by changing the word describing the element $y_7$, we would get a different barcode (for example $[-3, 4, 1, -2, -1, -1, 4, 1, 1, 2, -1, -4, 3, -4]$ is a simpler barcode for $y_7$), but this will not affect the result.

\vspace{10pt}
\begin{lstlisting}[language = python]
# Choose the genus and the nilpotency class

g=2
N=5


# The free associative algebra on 2g generators a1,...,ag,b1,...,bg

variables = ''
for i in range(g): variables = variables + 'a' + str(i+1) + ','
for i in range(g): variables = variables + 'b' + str(i+1) + ','
variables = variables[:6*g-1] 

A=FreeAlgebra(QQ,2*g,variables)
a=[A(1)]+[A.gen(i) for i in range(g)]
b=[A(1)]+[A.gen(i+g) for i in range(g)]


# "Fast" operations (product and bracket) in A up to degree N

def fpr(u,v):
	res = 0
	u = A(u)
	v = A(v)
	data_u =[(w.to_word(), cf) for (w,cf) in u]	
	data_v =[(w.to_word(), cf) for (w,cf) in v]
	for (wu,cfu)  in data_u:
		for (wv,cfv) in data_v:
			if len(wu) + len(wv) <= N:
				res = res + cfu*cfv*A.monomial(wu)*A.monomial(wv)	
	return res
	
def fbr(u,v):
	return fpr(u,v)-fpr(v,u)	


# Transforms a tensor into a cyclic tensor

def cyc(x):
	res = A(0)
	data = [(w.to_word(), cf) for (w,cf) in x]
	for (w, cf) in data:
		for i in range(len(w)):
			res = res + cf*A.monomial(Word(w[i:] + w[:i]))
	return res
    

# Extracts the degree k part

def extract(expr,k):
	data = [(w.to_word(), cf) for (w,cf) in expr if len(w)==k]
	return A.sum_of_terms(term for term in data)


# Truncates up to degree k

def truncate(expr,k):
	data = [(w.to_word(), cf) for (w, cf) in expr if len(w)<=k]
	return A.sum_of_terms(term for term in data)
    
    
# Computes the degree

def degree(expr):
	data = [len(w.to_word()) for (w,cf) in expr]
	return max(data)
     

# Computes, up to degree N, the exponential of a tensor without constant term

def exp(x):
	p= [A(1) for i in range(N+1)]
	for i in range(1,N+1):
		p[i] = fpr(p[i-1],x)
	res = A(0)
	for i in range(N+1):
		res = res + p[i]*(1/factorial(i))
	return truncate(res,N)
	
    
# Computes, up to degree N, the logarithm of a tensor whose constant term is one

def log(x):
	p = [A(1) for i in range(N+1)]
	d = x-1
	for i in range(1,N+1):
		p[i] = fpr(p[i-1],d)
	res = A(0)
	for i in range(1,N+1):
		res = res+p[i]*((-1)^(i+1))/i 
	return truncate(res,N)
    

# Values of a symplectic expansion "theta" up to order N

logtheta_a = [A(0)] + [a[i]-(1/2)*fbr(a[i],b[i])+(1/12)*fbr(fbr(a[i],b[i]),b[i])-(1/2)*fbr(sum(fbr(a[j],b[j]) for j in range(i)),a[i])
for i in range(1,g+1)]

logtheta_b = [A(0)] +  [b[i]-(1/2)*fbr(a[i],b[i])+(1/4)*fbr(fbr(a[i],b[i]),b[i])+(1/12)*fbr(a[i],fbr(a[i],b[i]))+(1/2)*fbr(b[i],sum(fbr(a[j],b[j]) for j in range(i)))
for i in range(1,g+1)]

theta_a = [exp(logtheta_a[i]) for i in range(g+1)]

theta_b = [exp(logtheta_b[i]) for i in range(g+1)]

theta_a_inv = [exp(-logtheta_a[i]) for i in range(g+1)]

theta_b_inv = [exp(-logtheta_b[i]) for i in range(g+1)]


# Computation of theta from a string such a 'a1+b2-a1-' which encodes an element of the fundamental group

def theta(lis):
	res = A(1)
	for j in range(len(lis)/3):
            index = int(lis[3*j+1])
            if [lis[3*j],lis[3*j+2]]==['a','+']: res = fpr(res,theta_a[index])
            if [lis[3*j],lis[3*j+2]]==['a','-']: res = fpr(res,theta_a_inv[index])
            if [lis[3*j],lis[3*j+2]]==['b','+']: res = fpr(res,theta_b[index])
            if [lis[3*j],lis[3*j+2]]==['b','-']: res = fpr(res,theta_b_inv[index])
	return truncate(res,N)
	
	
# Checks that this expansion is symplectic up to some degree

boundary = ''
for i in range(g): boundary = boundary + 'b' + str(i+1) + '-' + 'a' + str(i+1) + '+'  + 'b' + str(i+1) + '+' + 'a' + str(i+1) + '-'  

exp_omega_tilde = theta(boundary)
exp_omega = exp(sum( fbr(a[i],b[i]) for i in range(1,g+1)))
diff = exp_omega_tilde - exp_omega

print('Computations are done up to degree '+str(N)+'.')

d = 0
for i in range(N+1):
	if extract(diff,i)==0: d=i
print('The expansion is symplectic up to order '+str(d)+'.')


#  The map "L^theta_k" of Kawazumi & Kuno for a commutator

def Ltheta(lis,k):
	logtheta = log(theta(lis))
	ltheta = [extract(logtheta,i) for i in range(0,k-1)]
	res = A(0)
	for i in range(2,k-1):
		res = res + cyc(A(fpr(ltheta[i],ltheta[k-i])))
	return res*(1/2)
\end{lstlisting}
\begin{lstlisting}[language = python]
# Computation of the cyclic tensor corresponding a tree-like Jacobi diagram of the following form :
#
#    b  c  d
#    |  |  |
# a-----------e
#

def br(u,v):
	return u*v-v*u	

def treetotens(a,b,c,d,e):
	return cyc(br(a,b)*br(c,br(d,e)))


# Here we assume that g=2

# We represent elements of the fundamental group with barcodes or strings:
# Transforms a barcode such as [1,-2,3] to a string such as 'a1+b1-a2+'

def list_to_string(x):
    res=''
    for j in range(len(x)):
        if x[j] == 1 : res = res + 'a1+'
        if x[j] == -1 : res = res + 'a1-'
        if x[j] == 2 : res = res + 'b1+'
        if x[j] == -2 : res = res + 'b1-'
        if x[j] == 3 : res = res + 'a2+'
        if x[j] == -3 : res = res + 'a2-'
        if x[j] == 4 : res = res + 'b2+'
        if x[j] == -4 : res = res + 'b2-'
    return res


# Description of psi, by entering the expression (as barcodes) of the curves defining the twists of which psi is composed

def brack(a,b):
	return [a,b,-a,-b]

def bra(a,b):
	mira = list(reversed(a))
	mirb = list(reversed(b))
	return a+b+[-1*i for i in mira] + [-1*i for i in mirb]
        
gamma2 = list_to_string(brack(3,-4)+brack(1,-2))
t1 = list_to_string(bra(brack(-2,1)+[-4,1],[-2]))
t2 = list_to_string(bra([1],[-4,3,4,-2]))
t3 = list_to_string(bra([1],[-4,-3,4]+brack(1,-2)+[-2]))
t4 = list_to_string(bra([3],[-1,-4]))
t5 = list_to_string(bra([1],[-4,-3,-2]))
t6 = list_to_string(bra([3],[-2,-1,-4]))
t7 = list_to_string(bra([-3,4]+brack(1,-2)+[-2,-1,-4],[4]))
t8 = list_to_string(bra([3,4,1],[-2]))
t9 = list_to_string(bra([1],[-4,-2]))
t10 = list_to_string(bra([-4,-3,4]+brack(1,-2)+[-2],[4]))
t11 = list_to_string(bra(brack(-2,1)+[-4,3,4,1],[-2]))
t12 = list_to_string(bra([1,-4,-3,4]+brack(1,-2)+[4,1]+brack(-2,1)+[-4,3,4],[-4,-3,4]+brack(1,-2)+[-2]))
t13 = list_to_string(bra([-4,-3,4]+brack(1,-2)+[-2,-1],[1,2,4]))
s1 = list_to_string(brack(1,-2))
s2 = list_to_string(brack(3,-4))

listelem = [gamma2,t1,t2,t3,t4,t5,t6,t7,t8,t9,t10,t11,t12,t13,s1,s2]
listcoeff = [-3,-1,-1,+2,+2,+1,-1,-1,+1,-1,+1,-1,-1,+1,+7,+2]


# Computation of tau_2(psi) and tau_3(psi) 

tau2_psi = sum(listcoeff[i]*Ltheta(listelem[i],4) for i in range(16))
tau3_psi = sum(listcoeff[i]*Ltheta(listelem[i],5) for i in range(16))


# Comparison with the linear combinations of tree-like Jacobi diagrams

candidate = (-treetotens(a[2],a[1],a[1],b[1],a[1])-treetotens(a[2],b[1],a[1],a[2],a[1])-treetotens(b[2],a[1],a[1],b[1],a[1])-treetotens(b[2],b[1],a[1],b[1],a[1])+treetotens(b[2],a[2],a[1],b[1],a[1])+treetotens(b[2],a[2],a[1],a[2],a[1])+treetotens(b[2],a[2],a[1],b[2],a[1])+treetotens(b[2],a[2],b[1],b[2],a[1])+3*treetotens(b[2],a[2],a[2],b[1],a[1])+treetotens(b[2],a[2],a[2],a[2],a[1])+treetotens(b[2],a[2],b[2],b[1],a[1])-treetotens(b[2],a[1],a[2],b[1],a[1])+treetotens(b[2],b[1],a[2],b[1],a[1])+treetotens(b[2],a[2],b[2],a[2],a[1])-treetotens(b[2],a[2],b[2],a[2],b[1]))

# candidate is also equal to candidate_bis = (treetotens((b[1]+a[2]),a[1],(a[1]+a[2]+b[2]),a[2],(a[1]+b[2]))+treetotens(a[2]-a[1],b[2],a[1]+b[1],a[1],b[1]+b[2])-treetotens(a[2]-a[1],b[1],b[2],a[2],b[1]+b[2])+treetotens(b[2],a[2],(2*a[2]-2*a[1]+b[2]),b[1],a[1]))

print('psi is in J_3 : ' + str(tau2_psi == 0))
print('tau3_psi == candidate : ' + str(tau3_psi == candidate))
\end{lstlisting}

\label{appA}
\end{appendices}
\bibliographystyle{plain}
\bibliography{manuscrit}

\begin{thebibliography}{10}

\bibitem{bircra}
Joan~S. Birman and R.~Craggs.
\newblock The {$\mu $}-invariant of {$3$}-manifolds and certain structural
  properties of the group of homeomorphisms of a closed, oriented
  {$2$}-manifold.
\newblock {\em Trans. Amer. Math. Soc.}, 237:283--309, 1978.

\bibitem{faes}
Quentin Faes.
\newblock The handlebody group and the images of the second {J}ohnson
  homomorphism.
\newblock To appear in \emph{Algebraic \& Geometric Topology},
  arXiv:2001.09825, preprint 2021.

\bibitem{gous}
Mikhail Goussarov.
\newblock Finite type invariants and {$n$}-equivalence of {$3$}-manifolds.
\newblock {\em C. R. Acad. Sci. Paris S\'{e}r. I Math.}, 329(6):517--522, 1999.

\bibitem{masb}
Nathan Habegger and Gregor Masbaum.
\newblock The {K}ontsevich integral and {M}ilnor's invariants.
\newblock {\em Topology}, 39(6):1253--1289, 2000.

\bibitem{habclasper}
Kazuo Habiro.
\newblock Claspers and finite type invariants of links.
\newblock {\em Geom. Topol.}, 4:1--83, 2000.

\bibitem{mhsurvey}
Kazuo Habiro and Gw\'{e}na\"{e}l Massuyeau.
\newblock From mapping class groups to monoids of homology cobordisms: a
  survey.
\newblock In {\em Handbook of {T}eichm\"{u}ller theory. {V}olume {III}},
  volume~17 of {\em IRMA Lect. Math. Theor. Phys.}, pages 465--529. Eur. Math.
  Soc., Z\"{u}rich, 2012.

\bibitem{haininf}
Richard Hain.
\newblock Infinitesimal presentations of the {T}orelli groups.
\newblock {\em J. Amer. Math. Soc.}, 10(3):597--651, 1997.

\bibitem{joh80}
Dennis Johnson.
\newblock An abelian quotient of the mapping class group {I}.
\newblock {\em Mathematische annalen}, 249(3):225--242, 1980.

\bibitem{johquad}
Dennis Johnson.
\newblock Quadratic forms and the {B}irman-{C}raggs homomorphisms.
\newblock {\em Trans. Amer. Math. Soc.}, 261(1):235--254, 1980.

\bibitem{johsurvey}
Dennis Johnson.
\newblock A survey of the {T}orelli group.
\newblock {\em Contemporary Math}, 20:165--179, 1983.

\bibitem{joh2}
Dennis Johnson.
\newblock The structure of the {T}orelli group {I}{I}. {A} characterization of
  the group generated by twists on bounding curves.
\newblock {\em Topology}, 24(2):113--126, 1985.

\bibitem{kawa}
Nariya Kawazumi.
\newblock Cohomological aspects of magnus expansions, preprint 2005.

\bibitem{KK}
Nariya Kawazumi and Yusuke Kuno.
\newblock The logarithms of {D}ehn twists.
\newblock {\em Quantum Topol.}, 5(3):347--423, 2014.

\bibitem{levtrees}
Jerome Levine.
\newblock Labeled binary planar trees and quasi-lie algebras.
\newblock {\em Algebraic \& Geometric Topology}, 6(2):935--948, 2006.

\bibitem{masY2cl}
Gw\'{e}na\"{e}l Massuyeau.
\newblock Cohomology rings, {R}ochlin function, linking pairing and the
  {G}oussarov-{H}abiro theory of three-manifolds.
\newblock {\em Algebr. Geom. Topol.}, 3:1139--1166, 2003.

\bibitem{masinf}
Gw\'{e}na\"{e}l Massuyeau.
\newblock Infinitesimal {M}orita homomorphisms and the tree-level of the {LMO}
  invariant.
\newblock {\em Bull. Soc. Math. France}, 140(1):101--161, 2012.

\bibitem{masY3}
Gw{\'e}na{\"e}l Massuyeau and Jean-Baptiste Meilhan.
\newblock Equivalence relations for homology cylinders and the core of the
  {C}asson invariant.
\newblock {\em Transactions of the American Mathematical Society},
  365(10):5431--5502, 2013.

\bibitem{MasTur}
Gw\'{e}na\"{e}l Massuyeau and Vladimir Turaev.
\newblock Fox pairings and generalized {D}ehn twists.
\newblock {\em Ann. Inst. Fourier (Grenoble)}, 63(6):2403--2456, 2013.

\bibitem{mat87}
S.~V. Matveev.
\newblock Generalized surgeries of three-dimensional manifolds and
  representations of homology spheres.
\newblock {\em Mat. Zametki}, 42(2):268--278, 345, 1987.

\bibitem{mor}
Shigeyuki Morita.
\newblock Casson's invariant for homology 3-spheres and characteristic classes
  of surface bundles {I}.
\newblock {\em Topology}, 28(3):305--323, 1989.

\bibitem{mor91}
Shigeyuki Morita.
\newblock On the structure of the {T}orelli group and the {C}asson invariant.
\newblock {\em Topology}, 30(4):603--621, 1991.

\bibitem{mor93}
Shigeyuki Morita.
\newblock Abelian quotients of subgroups of the mapping class group of
  surfaces.
\newblock {\em Duke Mathematical Journal}, 70(3):699--726, 1993.

\bibitem{nss}
Yuta Nozaki, Masatoshi Sato, and Masaaki Suzuki.
\newblock Abelian quotients of the $ {Y} $-filtration on the homology cylinders
  via the {LMO} functor.
\newblock To appear in \emph{Geom. Topol.}, arXiv:2001.09825, preprint 2020.

\bibitem{pit}
Wolfgang Pitsch.
\newblock Integral homology 3-spheres and the {J}ohnson filtration.
\newblock {\em Transactions of the American Mathematical Society},
  360(6):2825--2847, 2008.

\end{thebibliography}
\address
\end{document}